\documentclass[12pt]{amsart}
\usepackage{amsmath, amssymb}
\usepackage{mathrsfs}
\newcommand{\no}[1]{#1}
\renewcommand{\no}[1]{}
\no{\usepackage{times}\usepackage[subscriptcorrection, slantedGreek, nofontinfo]{mtpro}
\renewcommand{\Delta}{\upDelta}}
\usepackage{color}
\usepackage{enumerate,comment}

\date{\today}
\newtheorem{Thm}{Theorem}[section]
\newtheorem{prop}{Proposition}[section]
\newtheorem{lem}{Lemma}[section]

\newtheorem{corollary}{Corollary}[section]

\theoremstyle{remark}

\newcommand{\bel}{\begin{equation} \label}
\newcommand{\ee}{\end{equation}}

\newcommand{\pd}{\partial}

\newcommand{\R}{{\mathbb R}}
\newcommand{\N}{{\mathbb N}}

\def\phi {\varphi}

\renewcommand{\leq}{\leqslant}
\renewcommand{\geq}{\geqslant}

\def\beq{\begin{equation}}
\def\eeq{\end{equation}}
\newcommand{\bea}{\begin{eqnarray}}
\newcommand{\eea}{\end{eqnarray}}
\newcommand{\beas}{\begin{eqnarray*}}
\newcommand{\eeas}{\end{eqnarray*}}

\providecommand{\abs}[1]{\left\lvert#1\right\rvert}
\providecommand{\norm}[1]{\left\lVert#1\right\rVert}

\numberwithin{equation}{section}

\title[An inverse problem for a quasilinear convection--diffusion equation]
{An inverse problem for a quasilinear convection--diffusion equation}
\author[A. Feizmohammadi]{Ali Feizmohammadi}
\address{The Fields Institute for Research in Mathematical Sciences, Toronto, Ontario, M5T 3J1, Canada}
\email{afeizmoh@fields.utoronto.ca}
\author[Y. Kian]{Yavar Kian}
\address{Aix Marseille Univ, Universit\'{e} de Toulon, CNRS, CPT, Marseille, France}
\email{yavar.kian@univ-amu.fr}
\author[G. Uhlmann]{Gunther Uhlmann}

\address
       {G. Uhlmann, Department of Mathematics\\
       University of Washington\\
       Seattle, WA  98195-4350\\
       USA\\
        and Institute for Advanced Study of the Hong Kong University of Science and Technology}
\email{gunther@math.washington.edu}


\begin{document}

\begin{abstract}
We study the inverse problem of recovering a semilinear diffusion term $a(t,\lambda)$ as well as a quasilinear convection term $\mathcal B(t,x,\lambda,\xi)$ in a nonlinear parabolic equation 
$$\partial_tu-\textrm{div}(a(t,u) \nabla u)+\mathcal B(t,x,u,\nabla u)\cdot\nabla u=0,  \quad \mbox{in}\ (0,T)\times\Omega,$$ 
given the knowledge of the flux of the moving quantity associated with different sources applied at the boundary of the domain. This inverse problem that is modeled by the solution dependent parameters $a$ and $\mathcal B$ has many physical applications related to various classes of cooperative interactions or complex mixing in diffusion processes. Our main result states that, under suitable assumptions, it is possible to fully recover the nonlinear diffusion term $a$ as well as the nonlinear convection term $\mathcal B$.  The recovery of the diffusion term is based on the idea of solutions to the linearized equation with singularities near the boundary $\partial \Omega$. Our proof of the recovery of the convection term is based on the idea of higher order linearization to reduce the inverse problem to a density property for certain anisotropic products of solutions to the linearized equation. We show this density property by constructing sufficiently smooth geometric optic solutions concentrating on rays in $\Omega$.
\end{abstract}
\maketitle


\section{Introduction}

Let $T>0$ and let $\Omega\subset \R^n$ with $n\geq 2$ be a bounded domain with a smooth boundary. We denote by $\nu(x)$ the outward unit normal to $\partial\Omega$ computed at $x \in \partial\Omega$.  Then, for $\lambda\in\R$,  we introduce the initial boundary value problem (IBVP in short)
\bel{eq1}
\left\{
\begin{array}{ll}
\partial_tu-\textrm{div}(a(t,u) \nabla u)+\mathcal B(t,x,u,\nabla u)\cdot\nabla u=0  & \mbox{in}\ (0,T)\times\Omega:=M ,
\\
u=\lambda+f &\mbox{on}\ (0,T)\times\partial\Omega:=\Sigma,\\
u(0,x)=\lambda &\ x\in\Omega.
\end{array}
\right.
\ee
Throughout this paper, we make the standing assumption that the nonlinear diffusion term $a\in\mathcal C^\infty([0,T]\times\R)$ satisfies
\bel{cond1}a(t,\lambda)>0,\quad (t,\lambda)\in[0,T]\times\R,\ee
and that the nonlinear convection term $\mathcal B\in\mathcal C^\infty([0,T]\times\overline{\Omega}\times\R\times\R^n)^n$
satisfies
\bel{form}\mathcal B(t,x,\tau,\xi)=b(t,x,\tau,\xi)\,B(t,x,\tau),\quad (t,x)\in (0,T)\times\Omega,\ (\tau,\xi)\in\R\times\R^n\ee
for some $B\in \mathcal C^\infty([0,T]\times \overline{\Omega}\times\R)^n$ and a scalar function $b\in \mathcal C^\infty([0,T]\times\overline{\Omega}\times\R\times\R^n)$ that satisfies
\bel{form1} 
b(t,x,\tau,0)=1.
\ee

From now on, we fix $\alpha\in(0,1)$ and we denote   by
$\mathcal C^{\frac{\alpha}{2},\alpha}([0,T]\times X)$, with $X=\overline{\Omega}$ or $X=\partial\Omega$,  the set of functions $h$ lying in  $ \mathcal C([0,T]\times  X)$ satisfying
$$[h]_{\frac{\alpha}{2},\alpha}=\sup\left\{\frac{|h(t,x)-h(s,y)|}{(|x-y|^2+|t-s|)^{\frac{\alpha}{2}}}:\ (t,x),(s,y)\in [0,T]\times X,\ (t,x)\neq(s,y)\right\}<\infty.$$
Then we define the space $\mathcal C^{1+\frac{\alpha}{2},2+\alpha}([0,T]\times X)$ as the set of functions $h$ lying in 
$$\mathcal C([0,T];\mathcal C^2(X))\cap \mathcal C^1([0,T];\mathcal C(X))$$ 
such that $$\partial_th,\partial_x^\beta h\in \mathcal C^{\frac{\alpha}{2},\alpha}([0,T]\times X),\quad \beta\in(\mathbb N\cup\{0\})^n,\ |\beta|=2.$$
We consider on these spaces the usual norms and we refer to \cite[pp. 4]{Ch} for more details.
 We introduce the space
$$\mathcal K_0:=\{h\in \mathcal C^{1+\frac{\alpha}{2},2+\alpha}([0,T]\times\partial\Omega):\ h(0,\cdot)=\partial_th(0,\cdot)= 0\}$$
and for all $r>0$ we denote by $\mathbb B_r$ the ball of center zero and of radius $r$ of the space $\mathcal K_0$.

As we will show in Proposition~\ref{TA1}, given any $\lambda\in\R$ there exists $\epsilon=\epsilon_{a,\mathcal B,\lambda}>0$, depending on  $a$, $\mathcal B$, $\lambda$, $\Omega$, $T$, such that, for $f\in \mathbb B_\epsilon$,  \eqref{eq1} admits a unique solution $u_\lambda\in\mathcal C^{1+\alpha/2,2+\alpha}([0,T];\mathcal C(\overline{\Omega}))$ that lies in a sufficiently small neighborhood of $\lambda$. We can define the parabolic Dirichlet-to-Neumann map
$$\mathcal N_{\lambda,a,\mathcal B}:\mathbb B_{\epsilon}\ni f\mapsto a(t,u_\lambda)\partial_\nu u_{\lambda}(t,x),\quad (t,x)\in(0,T)\times\partial\Omega.$$
Here the map $\mathcal N_{\lambda,a,\mathcal B}$ sends any small boundary source $\lambda+f$ located on the lateral boundary $(0,T)\times\partial\Omega$ to the associated measurement of the flux given by $a(t,u_\lambda)\partial_\nu u_{\lambda}$ that is measured  also on the lateral boundary. In this sense, the knowledge of the map $\mathcal N_{\lambda,a,\mathcal B}$ is equivalent to the knowledge of the flux for all possible Dirichlet excitation of the system on a neighborhood of the constant function $\lambda$.

Our inverse problem can now be posed as follows: Can we recover the nonlinear diffusion term $a$ and the  nonlinear convection term $\mathcal B$, given the knowledge of the parabolic Dirichlet-to-Neumann map $\mathcal N_{\lambda,a,\mathcal B}$ for all $\lambda\in \R$?  
\subsection{Motivations}
Let us recall that the equation in \eqref{eq1} can be associated with different class of nonlinear equations including nonlinear Fokker--Planck equations, nonlinear model of convection-diffusion equations and   multidimentional formulation of  generalized viscous Burgers' equations. Each of these equations are associated with different physical phenomenon. For instance, nonlinear Fokker--Planck equations of the form \eqref{eq1} have applications in various fields such as plasma physics, surface physics, astrophysics, physics of polymer fluids and particle beams, nonlinear hydrodynamics, population dynamics, human movement sciences and neurophysics. Here the fundamental physical mechanism  arises from cooperative interactions between the subsystems of many-body systems which leads to models described by nonlinear equations (see e.g. \cite{F}). In the same way, nonlinear model of convection-diffusion equations of the form \eqref{eq1} can describe the transfer of physical quantities whose concentration is given by  the solution of \eqref{eq1}. In this context, the nonlinearity of the equation \eqref{eq1} describes models where the diffusivity $a$ and the velocity field $\mathcal B$ depend on the concentration of the moving quantities. Such phenomena may occur in the context of complex mixing phenomena such as the Rayleigh--B\'enard convection where the velocity field depends on the temperature. We mention also that the equation \eqref{eq1} can be seen as a  multidimentional formulation of a generalized viscous Burgers' equation modeling several physical phenomena in fluid mechanics and gas dynamics. Finally, we mention that  IBVP similar to \eqref{eq1} can be considered in the context of cooling process in the production of heavy plates made of steel where the heat conduction in the time leads to some class of nonlinear parabolic equations where the nonlinear terms are associated with temperature dependent parameters (see e.g. \cite{RSSP}).

For all these models and the associated physical phenomenon, the goal of our inverse problem is to determine the nonlinear physical law of the system associated with \eqref{eq1}. This problem can be formulated in terms of simultaneous determination of the nonlinear diffusive  (or viscosity in the context of Burgers' equation) term $a(t,u)$ and of the nonlinear convection term $\mathcal B(t,x,u,\nabla u)$ modeling the drift vector for Fokker--Planck equations or  the velocity field of the moving quantity for convection-diffusion equations. 

Beside these physical motivations, there is also an important mathematical motivation for the study of such inverse problems due to their high nonlinearity. These problems can also be seen as a natural extension of similar problems of determination of coefficients stated for linear equations.

\subsection{Previous literature}
Inverse problems for various nonlinear equations have been widely studied in the last few decades. The key tool in the analysis of inverse problems for nonlinear equations is linearization of the PDE. In general, due to the presence of nonlinearity the solutions to the linearized equation can interact in a nonlinear fashion creating richer dynamics compared to the case of inverse problems for linear equations. This observation has been an underlying theme in majority of the works on inverse problems for nonlinear PDEs. The approach of first order linearization to solve inverse problems for a nonlinear equation was initiated by Isakov in \cite{Is1}. A second order linearization method was considered by Sun and Uhlmann in \cite{SuUh} while the idea of higher order linearization was fully utilized by Kurylev, Lassas and Uhlmann in \cite{KLU} to solve challenging inverse problems for hyperbolic equations. Without being exhaustive, we refer the reader for example to the works \cite{FLO,HUZ,Ki3,KLU,LLPT,LUW} that study inverse problems for nonlinear hyperbolic equations, \cite{CHY,FO20,ImYa,Is4,Is5,KU0,KU,LLLS,Sun2} for some results concerning semilinear elliptic equations as well as \cite{CF20,CFKKU,CNV,HS,Is2,KKU,MU,Sun1,SuUh} for results on quasilinear elliptic equations. All these works are based on the linearization method.  

In the context of nonlinear parabolic equations, the first results were concerned with recovery of semilinear terms $F(t,x,u)$ given the Dirichlet-to-Neumann map associated to the parabolic equation
$$\partial_t u - \Delta u + F(t,x,u)=0,\quad \text{on $(0,T)\times \Omega$}.$$
The recovery of nonlinearities of the form $F=F(u)$ was considered by Cannon and Yin in \cite{CY} and Pilant and Rundell in \cite{PR}, while the more general nonlinearity $F=F(x,u)$ was considred by Isakov in \cite{Is1}. There the author proved the recovery of time
independent semilinear terms of the form $F(x, u)$ given the additional over determination imposed by allowing arbitrary initial data  as well as final time overdetermination. The proof of \cite{Is1} is based on the first order linearization of the inverse problem combined with results of recovery of time-dependent coefficients proved by the same author in
\cite{Is0}. In \cite{CK} a further generalization of this result was derived together with a stability estimate. For further results in the semilinear parabolic setting with initial or final data over determination, we refer the reader to \cite{COY,Is2,Is3,KaRu1,KaRu2}. In the recent work \cite{KiUh} the authors considered the recovery of a general semilinear term depending on time variable, space variable and the solution and with zero initial conditions. 

The literature of studying inverse problems for quasilinear parabolic equations is rather sparse. We mention the work of Egger, Pietschmann and Schlottbom in \cite{EPS2} where in two and three dimensional physical space, the recovery of a semilinear term $a(t,u)$ was studied in the context of a quasilinear parabolic equation similar to \eqref{eq1}. However, the recovery of the nonlinear convection term was not considered there.

\subsection{Main results}
Our main result states that it is possible to uniquely determine the nonlinear diffusion term $a$ as well as the full Taylor series of the nonlinear convection term $\mathcal B$ at $\xi=0$, given the knowledge of the Dirichlet-to-Neumann map on a neighborhood of constant functions. Precisely, we will prove the following Theorem in Section~\ref{sec_proof_thm}. 
\begin{Thm}\label{t1}  For $j=1,2$, let $a_j\in\mathcal C^\infty([0,T]\times\R)$ satisfy \eqref{cond1} with $a=a_j$ and let $\mathcal B_j\in\mathcal C^\infty([0,T]\times\overline{\Omega}\times\R\times\R^n))$ satisfy \eqref{form}-\eqref{form1} with $B=B_j$, $b=b_j$. 
Then,  the condition
\bel{t1c} \mathcal N_{\lambda,a_1,\mathcal B_1}=\mathcal N_{\lambda,a_2,\mathcal B_2},\quad  \text{for all $\lambda \in \R$} \ee
implies that
\bel{t1d}a_1(t,\lambda)=a_2(t,\lambda),\quad t\in(0,T), \quad\lambda \in \R\ee
and
\bel{t1e}\partial_\xi^\beta \mathcal B_1(t,x,\lambda,0)=\partial_\xi^\beta \mathcal B_2(t,x,\lambda,0),\quad (t,x)\in(0,T)\times\Omega,\ \beta\in(\mathbb N\cup\{0\})^n,\, \lambda \in \R.\ee
\end{Thm}
As a direct consequence of Theorem \ref{t1}, we obtain the following results of full recovery of the parameter $a$ and $\mathcal B$.

\begin{corollary}\label{c1} Let the condition of Theorem \ref{t1} be fulfilled and assume that, for $j=1,2$ and all $(t,x,\lambda)\in (0,T)\times\Omega\times\mathbb R$, the map $\mathbb R^n\ni\xi\mapsto  b_j(t,x,\lambda,\xi)$ is real-analytic. Then condition \eqref{t1c} implies that $a_1=a_2$ and $\mathcal B_1=\mathcal B_2$.
\end{corollary}

The proof of Theorem~\ref{t1} will be divided into several steps. We begin by constructing the nonlinear diffusion term $a$ through using solutions to the first order linearization of \eqref{eq1} with singular behavior near the boundary. Next, using first order linearization of the DN map together with standard Geometric Optic solutions to the first order linearization of \eqref{eq1} allows us to recover the nonlinear convection term $\mathcal B$ at $\xi=0$. The recovery of the full Taylor series of $\mathcal B$ at $\xi=0$ will be divided into two steps. First, we use the idea of higher order linearization to reduce the problem of recovering the Taylor series of $\mathcal B$ at $\xi=0$ to a density property for certain anisotropic products of solutions to the first order linearization of \eqref{eq1}. Second, we prove the density claim by using Geometric optic solutions with higher regularity. Our density claim can be stated as follows. In the following proposition, $\pi(m+1)$ stands for the set of all permutations of $\{1,\ldots,m+1\}$. Given any $j=1,\ldots,n$ the notation $\partial_{j}$ stands for the partial derivative with respect to $x^{j}$.

\begin{prop}\label{prop1}
	Let $a_0\in \mathcal C^{\infty}([0,T];\R^+) $ and $B_0 \in \mathcal C^{\infty}([0,T]\times\overline\Omega)^n$. Let $m\in \N$ and let $Q$ be a continuous function on $[0,T]\times \overline{\Omega}$ with values in the symmetric tensors of rank $m$. Suppose that
	\[
	\sum_{\ell_1,\ldots,\ell_{m+1}\in \pi(m+1)}\int_{(0,T)\times \Omega}\left(\sum_{j_1,\ldots,j_m=1}^nQ^{j_1,\ldots,j_m}\partial_{j_1}v_{\ell_1}\ldots \partial_{j_m}v_{\ell_m}\right) (B_0\cdot \nabla v_{\ell_{m+1}})\,v_{m+2}\,dt\,dx=0, 
	\]
	for all $v_{\ell}\in H^1(0,T;\mathcal C^2(\overline\Omega))$, $\ell=1,\ldots,m+1$ solving
	\bel{prop1aa}\partial_tv_{\ell}-a_0(t) \Delta v_{\ell}+B_0(t,x)\cdot\nabla v_{\ell}=0,\quad \text{on $(0,T)\times \Omega$},\ee
	subject to $v_\ell(0,x)=0$ on $\Omega$ and all $v_{m+2}\in H^1(0,T;\mathcal C^2(\overline\Omega))$ solving
		\bel{prop1ab}-\partial_tv_{m+2}-a_0(t)\Delta v_{m+2}-\textrm{div}(B_0(t,x)v_{m+2})=0,\quad \text{on $(0,T)\times \Omega$},\ee
	subject to $v_{m+2}(T,x)=0$ on $\Omega$. Then, $Q\otimes B_0$ vanishes identically on $(0,T)\times\Omega$. 
	
\end{prop}	

\subsection{Comments about our results}

Let us first observe that to the best of our knowledge Theorem \ref{t1} and Corollary \ref{c1} are the first results for simultaneous  recovery of the two general classes of nonlinear terms $a$ and $\mathcal B$ satisfying \eqref{cond1}--\eqref{form}. The simultaneous recovery of these two classes of parameters relies partly on the fact that the diffusion term $a$ can be determined independently of the choice of the convection term $\mathcal B$. While this idea was already used by \cite{EPS2} in order to recover a nonlinear diffusion term in dimensions two and three, this article is the first in proving the simultaneous recovery of both these classes of parameters. Moreover, the nonlinear convection term in this paper has dependence not only on space and time but also on the solution and its gradient and as far as we know, even for $a\equiv1$, Theorem \ref{t1} and  Corollary \ref{c1} are the first results for the full recovery of such a general class of convection terms.

The recovery of  the diffusion term $a$ is based on the first order linearization and application of suitable singular solutions in the spirit of \cite{Alessandrini}. As observed by \cite{EPS2}, this approach allows us to determine the diffusion term $a$ independently of the choice of the convection term $\mathcal B$. Using this approach, we prove in Proposition \ref{t3} the unique recovery of the diffusion term $a$ given the data $\mathcal N_{\lambda,a,\mathcal B}$, $\lambda\in\R$. A similar problem was considered by \cite{EPS2} but with some extended knowledge of the parabolic Dirichlet-to-Neumann map not restricted to neighborhood of constant functions.

Once the unique recovery of the diffusion term $a$ is proved, we consider the determination of the nonlinear convection term $\mathcal B$. Here we use the higher order linearization approach initiated by \cite{KLU} in order to transform this inverse problem to a density property for solutions of the linearized problem as stated in Proposition \ref{prop1}. We prove Proposition \ref{prop1} by utilizing specific solutions of linear parabolic equations, called geometric optics solutions, that are constructed by means of suitable Carleman estimates. The construction of such solutions with constant second order coefficients can be found in \cite{CK1} with $L^2$-bounds on the remainder terms of the geometric optics solutions. However, in the context of Theorem \ref{t1} we need to consider such class of geometric optics solutions with second order time dependent coefficients and improved regularity. In Section~\ref{sec_smoothgo} we prove the construction of these new classes of geometric optics solutions that we design for the proof of Proposition \ref{prop1}. 

In order to prove Proposition~\ref{prop1}, we consider a specific class of geometric optics solutions whose products concentrate near arbitrary points $x_0\in\Omega$ and $t_0 \in (0,T)$. This idea is inspired by the approach of \cite{CFKKU} where a similar construction was carried out in the context of determining a nonlinear conductivity in an elliptic equation. Nevertheless, we would like to mention that for our parabolic problem there are several technical differences compared to the elliptic problem studied in \cite{CFKKU} that we will sketch as follows. One of these difficulties comes from the fact that the parabolic equation studied here is not self adjoint as opposed to the self adjoint elliptic equation studied in \cite{CFKKU}. This makes some of the symmetries present in the latter work to disappear as is apparent already from the statement of our Proposition~\ref{prop1} compared to the analogous proposition in \cite{CFKKU}. Secondly, the form of the geometric optics solutions here are rather different from the complex geometric optics solutions constructed in \cite{CFKKU}. This is mainly due to the parabolic scaling of the phase function, see \eqref{GO7}--\eqref{GO8}. As a consequence the process of canceling the exponential terms that are present in the geometric optics solutions is achieved via different arguments. 

\subsection{Organization of the paper}

This article is organized as follows. In Section 2, we show some properties of solutions of \eqref{eq1} including the well posedness for small data and the linearization properties. Section 3 is devoted to the unique recovery of the diffusive term $a(t,u)$ and the reduction of Theorem \ref{t1} into the density property of products of solutions of the linearized parabolic problem stated in Proposition \ref{prop1}. In Section 4, we introduce a new class of smooth Geometric optic solutions, with higher regularity, for some class of linear parabolic equation of the form \eqref{prop1aa}-\eqref{prop1ab}. Using the Geometric optic solutions of Section 4, in Section 5 we complete the proof of Proposition \ref{prop1} and by the same the proof of 
Theorem \ref{t1}.

\section{Preliminaries}

\subsection{Well-posedness for small data}

In this subsection, we consider the well posedness for the problem \eqref{eq1} when the data $f$ is sufficiently small. For this purpose, we consider the Banach space $\mathcal K_0$ with the norm of the space $\mathcal C^{1+\frac{\alpha}{2},2+\alpha}([0,T]\times \partial\Omega)$. Our result can be stated as follows.

\begin{prop}\label{TA1} Let $\mathcal B\in \mathcal C^\infty([0,T]\times\overline{\Omega}\times\R\times\R^n)^n$, $a\in\mathcal C^\infty([0,T]\times\R)$ satisfy condition \eqref{cond1}. Then for all $\lambda\in\R$, there exists $\epsilon>0$ depending on  $a$, $\mathcal B$, $\lambda$, $\Omega$, $T$, such that, for $f\in \mathbb B_\epsilon$, problem \eqref{eq1} admits a unique solution $u_\lambda\in\mathcal C^{1+\frac{\alpha}{2},2+\alpha}([0,T]\times\overline{\Omega}))$ satisfying
\bel{TA1a}\norm{u_\lambda-\lambda}_{\mathcal C^{1+\frac{\alpha}{2},2+\alpha}([0,T]\times\overline{\Omega}))}\leq C\norm{f}_{\mathcal C^{1+\frac{\alpha}{2},2+\alpha}([0,T]\times\partial\Omega)}.\ee
\end{prop}

\begin{proof} Let us first observe that we can split $u_\lambda$ into two terms $u_\lambda=\lambda+v_\lambda$, where $v=v_\lambda$ solves
\bel{eq1000}
\left\{
\begin{array}{ll}
\partial_tv-\textrm{div}(a(t,v+\lambda) \nabla v)+\mathcal B(t,x,v+\lambda,\nabla v)\cdot\nabla v=0  & \mbox{in}\ (0,T)\times\Omega ,
\\
v=f &\mbox{on}\ (0,T)\times\partial\Omega,\\
v(0,x)=0 &\ x\in\Omega.
\end{array}
\right.
\ee
Therefore, it is enough for our purpose to show that there exists $\epsilon>0$ depending on  $a$, $\mathcal B$, $\lambda$, $\Omega$, $T$, such that, for $f\in \mathbb B_\epsilon$, problem \eqref{eq1000} admits a unique solution $v_\lambda\in\mathcal C^{1+\frac{\alpha}{2},2+\alpha}([0,T]\times\overline{\Omega})$ satisfying
\bel{TA1v}\norm{v_\lambda}_{\mathcal C^{1+\frac{\alpha}{2},2+\alpha}([0,T]\times\overline{\Omega})}\leq C\norm{f}_{\mathcal C^{1+\frac{\alpha}{2},2+\alpha}([0,T]\times\partial\Omega)}.\ee
For this purpose, we consider the spaces
$$\mathcal H_0:=\{u\in \mathcal C^{1+\frac{\alpha}{2},2+\alpha}([0,T]\times\overline{\Omega}):\ u_{|\{0\}\times\overline{\Omega}}\equiv 0,\ \partial_tu_{|\{0\}\times\partial\Omega}\equiv 0\},$$
$$\mathcal L_0:=\{F\in \mathcal C^{\frac{\alpha}{2},\alpha}([0,T]\times\overline{\Omega}):\ F_{|\{0\}\times\partial\Omega}\equiv 0\}.$$
Then, we introduce the map $\mathcal G$ from $\mathcal K_0\times\mathcal H_0$ to the space $\mathcal L_0\times\mathcal K_0$ defined by
$$\mathcal G: (f,v)\mapsto(\partial_tv-\textrm{div}(a(t,v+\lambda) \nabla v)+\mathcal B(t,x,v+\lambda,\nabla v)\cdot\nabla v, v_{|(0,T)\times\partial\Omega}-f).$$
We will define the solution of \eqref{eq100} by applying the implicit function theorem to the map $\mathcal G$. Using the fact that $\mathcal B$ and $a$ are smooth, it follows that the map $\mathcal G$ is $\mathcal C^\infty$ on $\mathcal K_0\times\mathcal H_0$. Moreover, we have $\mathcal G(0,0)=(0,0)$ and
$$\partial_v\mathcal G(0,0)w=(\partial_tw-a(t,\lambda)\Delta w+\mathcal B(t,x,\lambda,0)\cdot\nabla w, w_{|(0,T)\times\partial\Omega}).$$
In order to apply the implicit function theorem, we will prove that the map $\partial_v\mathcal G(0,0)$ is an isomorphism from $\mathcal H_0$ to $\mathcal L_0\times\mathcal K_0$. For this purpose, let us fix $(F,h)\in \mathcal L_0\times\mathcal K_0$ and let us consider the linear problem
\bel{eq101}
\left\{
\begin{array}{ll}
\partial_tw-a(t,\lambda)\Delta w+\mathcal B(t,x,\lambda,0)\cdot\nabla w=F(t,x)  & \mbox{in}\ (0,T)\times\Omega ,
\\
w=h &\mbox{on}\ (0,T)\times\partial\Omega,\\
w(0,x)=0 &\ x\in\Omega.
\end{array}
\right.
\ee
Applying \cite[Theorem 5.2, Chapter IV, page 320]{LSU}, we deduce that problem \eqref{eq101} admits a unique solution $w\in \mathcal H_0$ satisfying
$$\norm{w}_{\mathcal C^{1+\frac{\alpha}{2},2+\alpha}([0,T]\times\overline{\Omega})}\leq C(\norm{F}_{\mathcal C^{\frac{\alpha}{2},\alpha}([0,T]\times\overline{\Omega})}+\norm{h}_{\mathcal C^{1+\frac{\alpha}{2},2+\alpha}([0,T]\times\partial\Omega)}).$$
 From this result we deduce that $\partial_v\mathcal G(0,0)$ is an isomorphism from $\mathcal H_0$ to $\mathcal L_0\times\mathcal K_0$. Therefore, applying the implicit function theorem, we deduce that there exists $\epsilon>0$ depending on  $a$, $\mathcal B$, $\lambda$, $\Omega$, $T$, and a smooth map $\psi$ from $\mathbb B_\epsilon$ to $\mathcal H_0$, such that, for all $f\in \mathbb B_\epsilon$, we have 
$\mathcal G(f,\psi(f))=(0,0)$.
This proves that , for all $f\in \mathbb B_\epsilon$, $v=\psi(f)$ is a solution of \eqref{eq1000}. Recalling that a solution of the problem   \eqref{eq1000} can also be seen as a solution of the linear problem with sufficiently smooth coefficients depending on $v$, we can apply \cite[Theorem 5.2, Chapter IV, page 320]{LSU} in order to deduce that $v=\psi(f)$ is the unique solution of \eqref{eq1000}. Combining this with the fact that $\psi$ is smooth from $B_\epsilon$ to $\mathcal H_0$, we obtain \eqref{TA1v}. This completes the proof of the theorem.\end{proof}

\subsection{Linearization of the problem}
Let $\mathcal B\in\mathcal C^\infty([0,T]\times\overline{\Omega}\times\R\times\R^n)^n$, $a\in\mathcal C^\infty([0,T]\times\R)$ satisfy condition \eqref{cond1}. Let us introduce $\lambda\in\R$, $m\in\mathbb N\cup\{0\}$ and consider $s=(s_1,\ldots,s_{m+1})\in(0,1)^{m+1}$ and $\lambda\in\R$. Fixing $g_1,\ldots,g_{m+1}\in \mathbb B_{\frac{\epsilon}{m+1}}$,
we consider $u=u_{s}$ the solution of
\bel{eq2}
\left\{
\begin{array}{ll}
\partial_tu-\textrm{div}(a(t,u) \nabla u)+\mathcal B(t,x,u,\nabla u)\cdot\nabla u=0  & \mbox{in}\ (0,T)\times\Omega ,
\\
u=\lambda+\sum_{i=1}^ms_ig_i &\mbox{on}\ (0,T)\times\partial\Omega,\\
u(0,x)=\lambda &\ x\in\Omega.
\end{array}
\right.
\ee
Following the proof of Proposition \ref{TA1}, we know that the map $s\mapsto u_s$ is lying in\\
 $\mathcal C^\infty\left((-1,1)^{m+1};\mathcal C^{1+\frac{\alpha}{2},2+\alpha}([0,T]\times\overline{\Omega})\right)$.
We will start by considering the partial derivative 
\bel{mthderivative_rep}
\partial_{s_1}\partial_{s_2}\ldots\partial_{s_{m+1}}u_s,\quad \text{at $s=0$}.
\ee
We can split  $u_{s}$ into $m+2$ terms 
$$ u_{s}=\lambda+s_1w_{1,s}+\ldots+s_{m+1}w_{m+1,s}.$$
where, for $\ell=1,\ldots,m$, $w_{\ell,s}\in  \mathcal C^{2+\alpha}(\overline{\Omega})$ solves 
\bel{eq3}
\left\{
\begin{array}{ll}
\partial_t w_{\ell,s}-\textrm{div}(a(t,u_s) \nabla w_{\ell,s})+\mathcal B(t,x,u_s,\nabla u_s)\cdot\nabla  w_{\ell,s}=0  & \mbox{in}\ (0,T)\times\Omega ,
\\
 w_{\ell,s}=g_\ell &\mbox{on}\ (0,T)\times\partial\Omega,\\
 w_{\ell,s}(0,x)=0 &\ x\in\Omega.
\end{array}
\right.
\ee
Our aim in the remainder of this section is to use the above representation formula for $u_{s}$, to justify and evaluate \eqref{mthderivative_rep}.

Let us first consider $\underset{s\to0}{\lim}\,w_{\ell,s}$, $\ell=1,\ldots,m+1$. For this purpose, we introduce
 the solution of the linear problem
\bel{eq4}
\left\{
\begin{array}{ll}
\partial_tv-a(t,\lambda)\Delta v+\mathcal B(t,x,\lambda,0)\cdot\nabla v=0,  & \mbox{in}\ (0,T)\times\Omega ,
\\
v=g &\mbox{on}\ (0,T)\times\partial\Omega,\\
 v(0,x)=0 &\ x\in\Omega
\end{array}
\right.
\ee

\begin{lem}\label{l1} For  $\ell=1,\ldots,m+1$, we consider $v_{\ell}$ the solution of \eqref{eq4} with    $g=g_\ell$. Then we have

\bel{l1a}\lim_{s\to0}\norm{w_{\ell,s}-v_{\ell}}_{\mathcal C^{1+\frac{\alpha}{2},2+\alpha}([0,T]\times\overline{\Omega})}=0.\ee
\end{lem}

\begin{proof} In all this proof $C$ and $\tilde C$ will be two generic constants depending on $a$,  $\mathcal B$, $\Omega$, $\lambda$, $\epsilon$ and $T$ that may change from line to line.
In view of Proposition \ref{TA1}, we know that \eqref{eq2}  admits a unique solution  $u_{s}\in  \mathcal H_0$ satisfying 
\bel{l1bb}\norm{u_{s}}_{\mathcal C^{1+\frac{\alpha}{2},2+\alpha}([0,T]\times\overline{\Omega})}\leq \tilde C.\ee
 Applying this estimate and fixing 
$$ a_s(t,x)=a(t,u_{s}(t,x)),\quad t\in(0,T),\ x\in\Omega,$$
$$\mathcal B_{s}(t,x)=\mathcal B(t,x,u_{s}(t,x),\nabla u_{s}(t,x)),\quad t\in(0,T),\ x\in\Omega,$$
we deduce that $a_s\in \mathcal C^{1+\frac{\alpha}{2},2+\alpha}([0,T]\times\overline{\Omega})$, $\mathcal B_s\in \mathcal C^{\frac{\alpha}{2},\alpha}([0,T]\times\overline{\Omega})$ with
$$\begin{aligned}\norm{a_s}_{\mathcal C^{1+\frac{\alpha}{2},2+\alpha}([0,T]\times\overline{\Omega})}&\leq C\norm{a(\cdot,0)}_{\mathcal C^2([0,T])}+C\sup_{k=0,\ldots,3}\sup_{|\tau|\leq \tilde C}\norm{\partial_\tau^k  a(\cdot,\tau,)}_{\mathcal C^1([0,T])}\norm{u_{s}}_{\mathcal C^{1+\frac{\alpha}{2},2+\alpha}([0,T]\times\overline{\Omega})}\\
&\leq C(1+\tilde C),\end{aligned}$$

\begin{multline}\label{l1aa}
\norm{\mathcal B_s}_{\mathcal C^{\frac{\alpha}{2},\alpha}([0,T]\times\overline{\Omega})} \leq C\norm{\mathcal B(\cdot,0,0)}_{\mathcal C^{\frac{\alpha}{2},\alpha}([0,T]\times\overline{\Omega})}\\
+C\sup_{|(k,\alpha)|\leq1}\sup_{|\tau|\leq M}\sup_{|\xi|\leq \tilde C}\norm{\partial_\tau^k \partial_{\xi}^\alpha \mathcal B(\cdot,\tau,\xi)}_{\mathcal C^{\frac{\alpha}{2},\alpha}([0,T]\times\overline{\Omega})}(\norm{u_{s}}_{\mathcal C^{\frac{\alpha}{2},\alpha}([0,T]\times\overline{\Omega})}+\norm{\nabla u_{s}}_{\mathcal C^{\frac{\alpha}{2},\alpha}([0,T]\times\overline{\Omega})})\\
\leq C(1+\norm{u_{s}}_{\mathcal C^{1+\frac{\alpha}{2},2+\alpha}([0,T]\times\overline{\Omega})})\leq C.
\end{multline}
 Therefore, in view of \cite[Theorem 5.2, Chapter IV, page 320]{LSU}, $w_{\ell,s}\in  \mathcal C^{1+\frac{\alpha}{2},2+\alpha}([0,T]\times\overline{\Omega})$ is the unique solution of 
$$\left\{
\begin{array}{ll}
\partial_t w_{\ell,s}-\textrm{div}(a_s(t,x) \nabla w_{\ell,s})+\mathcal B_s(t,x)\cdot\nabla  w_{\ell,s}=0  & \mbox{in}\ (0,T)\times\Omega ,
\\
 w_{\ell,s}=g_\ell &\mbox{on}\ (0,T)\times\partial\Omega,\\
 w_{\ell,s}(0,x)=0 &\ x\in\Omega
\end{array}
\right.$$
and it satisfies
\bel{l1b}\norm{w_{\ell,s}}_{\mathcal C^{1+\frac{\alpha}{2},2+\alpha}([0,T]\times\overline{\Omega})}\leq \tilde C.\ee
 Applying the mean value theorem, we deduce that
$$\begin{aligned}&\norm{ \mathcal B(\cdot,u_s(\cdot),\nabla u_s(\cdot)))- \mathcal B(\cdot,\lambda,0)}_{\mathcal C^{\frac{\alpha}{2},\alpha}([0,T]\times\overline{\Omega})}\\
&\leq C\sup_{k=1,2}\sup_{\tau,t\in[-|\lambda|-m\tilde C,|\lambda|+m\tilde C]}\left(\norm{D_{\tau,\xi}^k \mathcal B(\cdot,\tau,\xi)}_{\mathcal C^1([0,T]\times\overline{\Omega})}\right)\norm{s_1w_{1,s}+\ldots+s_{m+1}w_{m+1,s}}_{\mathcal C^1([0,T]\times\overline{\Omega})}.\end{aligned}$$
Combining this with \eqref{l1b}, we deduce that
\bel{l1c}\begin{aligned}&\norm{ \mathcal B_s- \mathcal B(\cdot,\lambda,0)}_{\mathcal C^{\frac{\alpha}{2},\alpha}([0,T]\times\overline{\Omega})}\\
&\leq C|s|.\end{aligned}\ee
In the same way, we prove that
\bel{l1ca}\norm{ a_s- a(\cdot,\lambda)}_{\mathcal C^{1+\frac{\alpha}{2},2+\alpha}([0,T]\times\overline{\Omega})}\leq C|s|.\ee
 Therefore, fixing  $y_{\ell,s}=v_{\ell}-w_{\ell,s}$, we deduce that $y=y_{\ell,s}$ solves the linear problem
\bel{l1d}
\left\{
\begin{array}{ll}\partial_ty-a(t,\lambda)\Delta y+\mathcal B(t,x,\lambda,0)\cdot\nabla y=K_{\ell,s},  & \mbox{in}\ (0,T)\times\Omega ,
\\
y=0 &\mbox{on}\ (0,T)\times\partial\Omega,\\
 y(0,x)=0 &\ x\in\Omega
\end{array}
\right.
\ee
with
$$K_{\ell,s}=[\mu_s-\mu(\cdot,\lambda)]\partial_tw_{\ell,s}-\textrm{div}([a_s-a(\cdot,\lambda)] \nabla w_{\ell,s}) \left[\mathcal B_s- \mathcal B(\cdot,\lambda,0)\right]\nabla w_{\ell,s}.$$
Applying \eqref{l1b}-\eqref{l1ca}, we deduce that
$$\begin{aligned}&\norm{K_{\ell,s}}_{\mathcal C^{\frac{\alpha}{2},\alpha}([0,T]\times\overline{\Omega})}\\
&\leq C\left[\norm{ a_s- a(\cdot,\lambda)}_{\mathcal C^{1+\frac{\alpha}{2},2+\alpha}([0,T]\times\overline{\Omega})}\right]\norm{w_{\ell,s}}_{\mathcal C^{1+\frac{\alpha}{2},2+\alpha}([0,T]\times\overline{\Omega})}\\
&\ \ \ +\norm{ \mathcal B_s- \mathcal B(\cdot,\lambda,0)}_{\mathcal C^{\frac{\alpha}{2},\alpha}([0,T]\times\overline{\Omega})}\norm{w_{\ell,s}}_{\mathcal C^{1+\frac{\alpha}{2},2+\alpha}([0,T]\times\overline{\Omega})}\\
&\leq C|s|.\end{aligned}$$
 Therefore, applying \cite[Theorem 5.2, Chapter IV, page 320]{LSU}, we obtain that
 $$\norm{w_{\ell,s}-v_{\ell}}_{\mathcal C^{1+\frac{\alpha}{2},2+\alpha}([0,T]\times\overline{\Omega})}=\norm{y_{\ell,s}}_{\mathcal C^{1+\frac{\alpha}{2},2+\alpha}([0,T]\times\overline{\Omega})}\leq C\norm{K_{\ell,s}}_{\mathcal C^{\frac{\alpha}{2},\alpha}([0,T]\times\overline{\Omega})}\leq C|s|$$
which implies  \eqref{l1a}.\end{proof}

Formula \eqref{l1a} gives us a limit of the expression $w_{\ell,s}$ as $s\to0$. Let us now consider, for $N=1,\ldots,m+1$, the partial derivative $\partial_{s_N}u_{s}$ at $s=0$. 
\begin{lem}\label{l2} For $\ell,N=1,\ldots,m+1$, $\partial_{s_N}u_{s}|_{s=0}$ is well defined and we have
\bel{l2a}\partial_{s_N}u_{s}|_{s=0}= v_{N}\ee
in the sense of functions taking values in $\mathcal C^{1+\frac{\alpha}{2},2+\alpha}([0,T]\times\overline{\Omega})$.
\end{lem}
\begin{proof} Using the fact that $s\mapsto u_s$ is $\mathcal C^\infty$ on some neighborhood of $s=0$ as functions taking values in $\mathcal C^{1+\frac{\alpha}{2},2+\alpha}([0,T]\times\overline{\Omega})$, we deduce that $s\mapsto a_s$ (resp.  $s\mapsto \mathcal B_s$) are $\mathcal C^\infty$ on a neighborhood of $s=0$ as functions taking values in $\mathcal C^{1+\frac{\alpha}{2},2\alpha}([0,T]\times\overline{\Omega})$ (resp.  $\mathcal C^{\frac{\alpha}{2},\alpha}([0,T]\times\overline{\Omega})^n$). 
Combining this with \eqref{l1a}, we deduce that $u_s|_{s=0}=\lambda$ which implies
\bel{l2b}\partial_{s_N}[\partial_t u_s]|_{s=0}=\partial_t (\partial_{s_N}u_s|_{s=0})\ee
where the derivative is considered in terms of functions taking values in $\mathcal C^{\frac{\alpha}{2},\alpha}([0,T]\times\overline{\Omega})$. 
In the same way, we obtain
\bel{ll2c}\partial_{s_N}[-\textrm{div}(a(t,u_s) \nabla u_s)]|_{s=0}=-a(t,\lambda)\Delta (\partial_{s_N}u_s|_{s=0}),\ee
\bel{l2dd}\partial_{s_N}[\mathcal B_s \cdot\nabla u_s]|_{s=0}=\mathcal B(t,x,\lambda,0)\cdot\nabla (\partial_{s_N}u_s|_{s=0}).\ee
Combining \eqref{l2b}-\eqref{l2dd}, we deduce that $z_N=\partial_{s_N}u_s|_{s=0}$ solves the IBVP
$$\left\{
\begin{array}{ll}
\partial_t z_N-a(t,\lambda)\Delta z_N+\mathcal B(t,x,\lambda,0)\cdot\nabla z_N=0  & \mbox{in}\ (0,T)\times\Omega ,
\\
z_N=g_N &\mbox{on}\ (0,T)\times\partial\Omega,\\
z_N(0,x)=0 &\ x\in\Omega
\end{array}
\right.$$
Then the uniqueness of the solution of the above IBVP implies that $\partial_{s_N}u_{s}|_{s=0}=z_N= v_{N}$. Finally, combining the fact that the identities \eqref{l2b}-\eqref{l2dd} hold true as functions taking values in $\mathcal C^{\frac{\alpha}{2},\alpha}([0,T]\times\overline{\Omega})$ with \cite[Theorem 5.2, Chapter IV, page 320]{LSU} and the arguments used in the proof of Lemma \ref{l1}, we deduce that \eqref{l2a} holds true in the sense of functions taking values in $\mathcal C^{1+\frac{\alpha}{2},2+\alpha}([0,T]\times\overline{\Omega})$.

\end{proof}

Now let us turn to the expression $\partial_{s_{\ell_1}}\partial_{s_{\ell_2}}u_{j,s}|_{s=0}$
For this purpose, we introduce the function
$w_{\ell_1,\ell_2}\in \mathcal C^{1+\frac{\alpha}{2},2+\alpha}([0,T]\times\overline{\Omega})$ solving the linear problem
\bel{eq5}
\left\{
\begin{array}{ll}
\partial_t w_{\ell_1,\ell_2}-a(t,\lambda)\Delta w_{\ell_1,\ell_2}+\mathcal B(t,x,\lambda,0)\cdot\nabla w_{\ell_1,\ell_2}=H^{(1)}_{\ell_1,\ell_2}  & \mbox{in}\ (0,T)\times\Omega,
\\
w_{\ell_1,\ell_2}=0 &\mbox{on}\ (0,T)\times\partial\Omega,\\
 w_{\ell_1,\ell_2}(0,x)=0 &\ x\in\Omega,
\end{array}
\right.
\ee
where
$$\begin{aligned}H^{(1)}_{\ell_1,\ell_2}=&-\partial_\tau [v_{\ell_1}\partial_tv_{\ell_2}+v_{\ell_2}\partial_tv_{\ell_1}]+\partial_\tau a(t,\lambda)\textrm{div}[v_{\ell_1}\nabla v_{\ell_2}+v_{\ell_2}\nabla v_{\ell_1}]\\
\ &-(\partial_\tau \mathcal B(t,x,\lambda,0)\cdot\nabla v_{\ell_1}) v_{\ell_2}-(\partial_\tau \mathcal B(t,x,\lambda,0)\cdot\nabla v_{\ell_2}) v_{\ell_1}\\
\ &-(\partial_\xi \mathcal B(t,x,\lambda,0)\nabla v_{\ell_1})\cdot\nabla v_{\ell_2}-(\partial_\xi \mathcal B(t,x,\lambda,0)\nabla v_{\ell_2})\cdot\nabla v_{\ell_1}.\end{aligned}$$
Repeating the arguments of Lemma \ref{l2}, we obtain the following.

\begin{lem}\label{l3} For $\ell,N=1,\ldots,m+1$, $\partial_{s_N}w_{\ell,s}|_{s=0}$ is well defined and we have
\bel{l3a}\partial_{s_N}w_{\ell,s}|_{s=0}= w_{N,\ell}\ee
in the sense of function taking values in $\mathcal C^{1+\frac{\alpha}{2},2+\alpha}([0,T]\times\overline{\Omega})$.
\end{lem}

We can prove by iteration the following result.

\begin{lem}\label{l5} The function 
	\bel{l5a}w^{(m+1)}=\partial_{s_1}\partial_{s_2}\ldots\partial_{s_{m+1}}u_{s}|_{s=0}\ee
	is well defined in the sense of functions taking values in $\mathcal C^{1+\frac{\alpha}{2},2+\alpha}([0,T]\times\overline{\Omega})$. Moreover, $w^{(m+1)}$ solves
	\bel{eq6}
	\left\{
	\begin{array}{ll}
		\partial_t w^{(m+1)}-a(t,\lambda)\Delta w^{(m+1)}+B(t,x,\lambda)\cdot\nabla w^{(m+1)}+H^{(m+1)}=0  & \mbox{in}\ (0,T)\times\Omega,
		\\
		w^{(m+1)}=0 &\mbox{on}\ (0,T)\times\partial\Omega,\\
		w^{(m+1)}(0,x)=0 &\ x\in\Omega,
	\end{array}
	\right.
	\ee
Here, (recalling that $\mathcal B$ has the special form \eqref{form}) we have 
\bel{Hm}
H^{(m+1)}=\sum_{\ell\in\pi(m+1)}\sum_{j_1,\ldots,j_{m}=1}^n\left((\partial_{\xi_{j_1}}\ldots\partial_{\xi_{j_{m}}}b|_{\xi=0})\partial_{j_1}v_{\ell_1}\ldots\partial_{j_{m}} v_{\ell_{m}}\right)B\cdot\nabla v_{\ell_{m+1}}+K^{(m+1)},\ee
where all the functions are evaluated at the point $(t,x)$ and $K^{(m+1)}(t,x)$ depends only on $a$, 
$\partial_\xi^k\mathcal B(t,x,\lambda,0)$,  $k=0,\ldots,m-1$,  and $v_1,\ldots, v_{m+1}$.
\end{lem}

\section{Proof of Theorem~\ref{t1}}
\label{sec_proof_thm}

This section is concerned with the proof of Theorem~\ref{t1}. We start by showing that the Dirichlet-to-Neumann map $\mathcal N_{\lambda,a,\mathcal B}$ uniquely determines the diffusion term $a(t,\lambda)$ as well as the zeroth order term in the Taylor series of $\mathcal B$ at $(t,x,\lambda,0)$. This is achieved by studying the first order linearization of the Dirichlet-to-Neumann map.

We will then proceed to determine the remainder of the terms in the Taylor series of $\mathcal B$ at $(t,x,\lambda,0)$ by combining the idea of higher order linearizations of the Dirichlet-to-Neumann map together with the density property stated in Proposition~\ref{prop1}.  

\subsection{Recovery of the nonlinear diffusion term $a(t,\lambda)$}

In this section, applying the linearization procedure described in the preceeding section, we prove the recovery of the nonlinear term $a$  given the knowledge of $\mathcal N_{\lambda,a,\mathcal B}$, $\lambda\in\R$. More precisely, we prove the following.
\begin{prop}\label{t3}  Let the condition of Theorem \ref{t1} be fulfilled.
	Then, for any $\lambda\in\R$, the condition \eqref{t1c} implies that \eqref{t1d} holds true.
	
\end{prop}

We fix $\lambda\in\R$ and $g\in\mathcal H_0$ with $\norm{g}_{\mathcal C^{1+\frac{\alpha}{2},2+\alpha}([0,T]\times\partial\Omega)}<\epsilon$. We consider for $\tau\in(-1,1)$, $u_{j,\tau}$ solving
$$\left\{
\begin{array}{ll}
\partial_tu_{j,\tau}-\textrm{div}(a_j(t,u_{j,\tau})\nabla u)+\mathcal B_j(t,x,u_{j,\tau},\nabla u_{j,\tau})\cdot\nabla u_{j,\tau}=0  & \mbox{in}\ (0,T)\times\Omega ,
\\
u_{j,\tau}=\lambda+\tau g &\mbox{on}\ (0,T)\times\partial\Omega\\
u_{j,\tau}(0,x)=0 &\ x\in\Omega.
\end{array}
\right.$$

In a similar way to Lemma \ref{l2}, we can prove that $w_j=\partial_\tau u_{j,\tau}|_{\tau=0}$ solves 
\bel{eq100}\left\{
\begin{array}{ll}
	\partial_tw_j-a_j(t,\lambda) \Delta w_j+\mathcal B_j(t,x,\lambda,0)\cdot\nabla w_j=0  & \mbox{in}\ (0,T)\times\Omega ,
	\\
	w_j= g &\mbox{on}\ (0,T)\times\partial\Omega\\
	w_j(0,x)=0 &\ x\in\Omega.
\end{array}
\right.\ee
Therefore, fixing 
$$\Lambda_{j,\lambda}:\mathcal H_0\ni g\mapsto a_j(t,\lambda)\partial_\nu w_j|_{(0,T)\times\partial\Omega}$$
we deduce from \eqref{t1c} that
\bel{t3d}\Lambda_{1,\lambda}=\Lambda_{2,\lambda}.\ee
In view of this result the proof of Proposition \ref{t3} will be completed if we can prove the following.

\begin{lem}\label{l100} Let  the condition  \eqref{t3d} be fulfilled. Then the condition \eqref{t1d} is fulfilled.
\end{lem}
\begin{proof} We prove this result by applying an approach based on singular solutions inspired by  \cite{Alessandrini} in our specific context with solutions of parabolic equations and any dimension of space $n\geq2$. We start by proving this result for $n\geq3$. 
	We will prove  that \eqref{t3d} implies that
	\bel{t3e}a_1(t,\lambda)=a_2(t,\lambda),\quad t\in(0,T).\ee
	For this purpose, we will proceed by contradiction. Let us assume that \eqref{t3d} is fulfilled but \eqref{t3e} is not fulfilled. Then, without loss of generality, we may assume that there exists $0<t_0<t_1<T$, such that 
	\bel{t3f}a_1(t,\lambda)<a_2(t,\lambda),\quad t\in[t_0,t_1].\ee
	Fix $r>0$ and $y\in\R^n\setminus\overline{\Omega}$ such that dist$(y,\Omega)=r$. Consider also $\Phi_y\in\mathcal C^\infty(\overline{\Omega})$ defined by 
	$$\Phi_y(x)=\frac{1}{n(2-n)d_n}|x-y|^{2-n},\quad x\in\overline{\Omega},$$
	where $d_n$ denotes the volume of the unit ball in $\R^n$.
	Fix also $\delta \in(0,(t_1-t_0)/4)$, $\chi\in\mathcal C^\infty_0(t_0,t_1)$, satisfying $0\leq\chi\leq 1$ and $\chi=1$ on $[t_0+\delta,t_1-\delta]$. Then, we set
	$$\Psi_y(t,x)=\chi(t)\Phi_y(x),\quad x\in\overline{\Omega},\ t\in[0,T].$$
	For $j=1,2$, let $w_j$ be the solution of
	$$\left\{
	\begin{array}{ll}
	\partial_tw_j-a_j(t,\lambda) \Delta w_j+\mathcal B_j(t,x,\lambda,0)\cdot\nabla w_j=0  & \mbox{in}\ (0,T)\times\Omega ,
	\\
	w_j= \Psi_y &\mbox{on}\ (0,T)\times\partial\Omega\\
	w_j(0,x)=0 &\ x\in\Omega
	\end{array}
	\right.$$
	and $w_j^*$, $j=1,2$, the solution of the adjoint system
	$$\left\{
	\begin{array}{ll}
	-\partial_tw_j^*-a_j(t,\lambda) \Delta w_j^*-\textrm{div}\left(\mathcal B_j(x,\lambda,0) w_j^*\right)=0  & \mbox{in}\ (0,T)\times\Omega ,
	\\
	w_j^*= \Psi_y &\mbox{on}\ (0,T)\times\partial\Omega\\
	w_j^*(T,x)=0 &\ x\in\Omega.
	\end{array}
	\right.$$
	Recall that $\Delta \Phi_y=0$ on $\Omega$. Using this property, we can split $w_j$ into $w_j=\Psi_y+z_j$ with $z_j$ solving
	$$\left\{
	\begin{array}{ll}
	\partial_tz_j-a_j(t,\lambda) \Delta z_j+\mathcal B_j(t,x,\lambda,0)\cdot\nabla z_j=G_j  & \mbox{in}\ (0,T)\times\Omega ,
	\\
	z_j= 0 &\mbox{on}\ (0,T)\times\partial\Omega\\
	z_j(0,x)=0 &\ x\in\Omega,
	\end{array}
	\right.$$
	with 
	$$G_j(t,x)=-\chi'(t)\Phi_y-\chi(t) \mathcal B_j(t,x,\lambda,0)\cdot\nabla \Phi_y(x).$$
	Applying \cite[Theorem 4.1., Chapter 3]{LM1} and \cite[Theorem 5.3., Chapter 4]{LM2}, we deduce that this problem admits a unique solution $z_j\in L^2(0,T;H^2(\Omega))\cap H^1(0,T;L^2(\Omega))$ satisfying the estimate
	\bel{fifi1}\begin{aligned}\norm{z_j}_{L^2(0,T;H^1(\Omega))}&\leq C\norm{G_j}_{L^2(0,T;H^{-1}(\Omega))}\\
		\ &\leq C\norm{\Phi_y}_{L^2(\Omega)}+ C\norm{\mathcal B_j(t,x,\lambda,0)\cdot\nabla \Phi_y}_{L^2(0,T;H^{-1}(\Omega))}\\
		\ &\leq C(1+\norm{\mathcal B_j(\cdot,\lambda,0)}_{L^\infty(0,T;W^{1,\infty}(\Omega))})\norm{\Phi_y}_{L^2(\Omega)}\\
		\ &\leq C\norm{\Phi_y}_{L^2(\Omega)}\end{aligned} \ee
	where $C$ is independent of $r$. On the other hand, fixing $R>0$ such that $\Omega$ is contained into $\mathbb B(y,R)=\{x\in\R^n:\ |x-y|<R\}$ and using the fact that dist$(y,\Omega)=r$, we get
	\bel{fifi}\norm{\Phi_y}_{L^2(\Omega)}^2=C\int_\Omega |x-y|^{4-2n}dx\leq Cr^{4-n-\frac{1}{8}}\int_{\mathbb B(y,R)} |x-y|^{\frac{1}{8}-n}dx=Cr^{4-n-\frac{1}{8}}\ee
	with $C>0$ independent of $r$. Combining this with \eqref{fifi1}, we find
	\bel{l100c} \norm{z_j}_{L^2(0,T;H^1(\Omega))}\leq Cr^{2-\frac{n}{2}-\frac{1}{16}},\ee
	with $C$ independent of $r$. In the same way, we can split $w_1^*$ into $w_1^*=\Psi_y+z_1^*$ with 
	\bel{l100d} \norm{z_1^*}_{L^2(0,T;H^1(\Omega))}\leq Cr^{2-\frac{n}{2}-\frac{1}{16}},\ee
	where $C$ is independent of $r$. We fix $w=w_1-w_2$ and we remark that $w$ solves
	$$\left\{
	\begin{array}{ll}
	\partial_tw-a_1(t,\lambda) \Delta w+\mathcal B_1(t,x,\lambda,0)\cdot\nabla w=K  & \mbox{in}\ (0,T)\times\Omega ,
	\\
	w= 0 &\mbox{on}\ (0,T)\times\partial\Omega\\
	w(0,x)=0 &\ x\in\Omega,
	\end{array}
	\right.$$
	with
	$$K(t,x)=(a_1(t,\lambda)-a_2(t,\lambda))\Delta w_2+\left[\mathcal B_2(t,x,\lambda,0)-\mathcal B_1(t,x,\lambda,0)\right]\cdot\nabla w_2.$$
	Applying condition $\Lambda_{1,\lambda}=\Lambda_{2,\lambda}$ and integrating by parts, we obtain
	\begin{multline}
	0=\left\langle \Lambda_{1,\lambda}\Psi_y-\Lambda_{2,\lambda}\Psi_y,\Psi_y\right\rangle_{L^2((0,T)\times\partial\Omega)}\\
	=\int_0^T\int_\Omega [a_1(t,\lambda)\Delta w_1w_1^*+a_1(t,\lambda)\nabla w_1\cdot\nabla w_1^*-a_2(t,\lambda)\Delta w_2w_1^*-a_2(t,\lambda)\nabla w_2\cdot\nabla w_1^*]dxdt\\
	=\int_0^T\int_\Omega [a_1(t,\lambda)\Delta ww_1^*+a_1(t,\lambda)\nabla w\cdot\nabla w_1^*+(a_1(t,\lambda)-a_2(t,\lambda))(\Delta w_2w_1^*+\nabla w_2\cdot\nabla w_1^*)]dxdt.
	\end{multline}
	Using the fact that supp$(\chi)\subset(t_0,t_1)$ and applying the uniqueness of solutions of parabolic IBVP, we deduce that, for $j=1,2$, $w_j=0$ on $(0,t_0)\times\Omega$ and $w_j^*=0$ on $(t_1,T)\times\Omega$. Thus, we obtain
	\[
	\int_{t_0}^{t_1}\int_\Omega [a_1(t,\lambda)\Delta ww_1^*+a_1(t,\lambda)\nabla w\cdot\nabla w_1^*+(a_1(t,\lambda)-a_2(t,\lambda))(\Delta w_2w_1^*+\nabla w_2\cdot\nabla w_1^*)]dxdt=0.
	\]
	On the other hand, we have
	\begin{multline*}a_1(t,\lambda)\Delta w=\partial_tw+\mathcal B_1(t,x,\lambda,0)\cdot\nabla w-(a_1(t,\lambda)-a_2(t,\lambda))\Delta w_2\\
	-[\mathcal B_2(t,x,\lambda,0)-\mathcal B_1(t,x,\lambda,0)]\cdot\nabla w_2
	\end{multline*}
	and it follows
	\begin{multline*}\int_{t_0}^{t_1}\int_\Omega \partial_t w\,w_1^*+[([\mathcal B_2(t,x,\lambda,0)-\mathcal B_1(t,x,\lambda,0)]\cdot\nabla w_2) w_1^*dxdt\\
	+\int_{t_0}^{t_1}\int_\Omega(\mathcal B_1(t,x,\lambda,0)\cdot\nabla w) w_1^*+a_1(t,\lambda)\nabla w\cdot\nabla w_1^*)\,dx\,dt\\
	+\int_{t_0}^{t_1}\int_\Omega (a_1(t,\lambda)-a_2(t,\lambda))\nabla w_2\cdot\nabla w_1^*]dxdt=0.\end{multline*}
	Finally, using the fact that $w_{|(0,T)\times\partial\Omega}\equiv0$, $w_{|(0,t_0)\times\Omega}\equiv0$, $w_1^*=0$ on $(t_1,T)\times\Omega$ and integrating by parts, we get
	$$\begin{aligned}&\int_{t_0}^{t_1}\int_\Omega \partial_t w\,w_1^*+\left(\mathcal B_1(t,x,\lambda,0)\cdot\nabla w\right) w_1^*+a_1(t,\lambda)\nabla w\cdot\nabla w_1^*dxdt\\
	&=\int_0^T\int_\Omega [-\partial_tw_1^*-a_1(t,\lambda) \Delta w_1^*-\textrm{div}\left(\mathcal B_1(t,x,\lambda,0) w_1^*\right)]wdxdt=0.\end{aligned}$$
	Therefore, we obtain
	$$\int_{t_0}^{t_1}(a_1(t,\lambda)-a_2(t,\lambda))\int_\Omega \nabla w_2\cdot\nabla w_1^*dx+\int_{t_0}^{t_1}\int_\Omega [\mathcal B(t,x)\cdot\nabla w_2]w_1^*dxdt=0,$$
	with 
	$$\mathcal B(t,x,\lambda)=\mathcal B_2(t,x,\lambda,0)-\mathcal B_1(t,x,\lambda,0).$$
	From this last identity, we deduce that
	\bel{l100e} \int_{t_0}^{t_1}(a_1(t,\lambda)-a_2(t,\lambda))\int_\Omega \nabla w_2\cdot\nabla w_1^*dx=-\int_{t_0}^{t_1}\int_\Omega [\mathcal B(t,x,\lambda)\cdot\nabla w_2]w_1^*dxdt.\ee
	Recall that
	$$\nabla \Phi_y(x)=\frac{1}{nd_n}\frac{x-y}{|x-y|^n},\quad x\in\Omega.$$
	Choosing $r>0$ sufficiently small, we can find $x_0\in\Omega$ such that $|x_0-y|=3r$ and $$\mathbb B(x_0,r):=\{x\in\R^n:\ |x-x_0|<r\}\subset \Omega.$$ 
	Therefore, we get
	$$\begin{aligned}\int_\Omega |\nabla \Phi_y(x)|^2dx&\geq \frac{1}{nd_n}\int_{\mathbb B(x_0,r)} |x-y|^{2-2n}dx\\
	\ &\geq \frac{1}{nd_n}\int_{\mathbb B(x_0,r)} (|x_0-y|-|x-x_0|)^{2-2n}dx\\
	\ &\geq \frac{1}{nd_n}\int_{\mathbb B(x_0,r)} (2r)^{2-2n}dx=\frac{2^{2-2n}}{n}r^nr^{2-2n}=\frac{2^{2-2n}r^{2-n}}{n}.\end{aligned}$$
	In the same way, we find
	$$\norm{\nabla \Phi_y}_{L^2(\Omega)}^2=C\int_\Omega |x-y|^{2-2n}dx\leq Cr^{2-n-\frac{1}{8}}\int_{\mathbb B(y,R)} |x-y|^{\frac{1}{8}-n}dx=Cr^{2-n-\frac{1}{8}},$$
	with $C>0$ independent of $r$.
	Combining these two estimates, we get
	\bel{l100r}  cr^{2-n}\leq\int_\Omega |\nabla \Phi_y(x)|^2dx\leq Cr^{2-n-\frac{1}{8}},\ee
	with $C,c>0$ independent of $r$. Moreover, we have
	\begin{multline*}\int_{t_0}^{t_1}\int_\Omega (a_2(t,\lambda)-a_1(t,\lambda))|\nabla \Psi_y(t,x)|^2dxdt\\
	=\left(\int_{t_0}^{t_1}(a_2(t,\lambda)-a_1(t,\lambda))|\chi(t)|^2dt\right)\left(\int_\Omega |\nabla \Phi_y(x)|^2dx\right)
	\end{multline*}
	and, since $\chi\not\equiv0$, we deduce that
	$$  c\inf_{t\in[t_0,t_1]}[a_2(t,\lambda)-a_1(t,\lambda)]r^{2-n}\leq\int_{t_0}^{t_1}\int_\Omega (a_1(t,\lambda)-a_2(t,\lambda))|\nabla \Psi_y(t,x)|^2dxdt,$$
	with $c>0$ independent of $r$. Combining this with \eqref{l100c}-\eqref{l100d} and the fact that $w_j=\Psi_y+z_j$, $j=1,2$, and $w_1^*=\Psi_y+z_1^*$, we obtain that for $r>0$ sufficiently small, we have
	$$c\inf_{t\in[t_0,t_1]}[a_2(t,\lambda)-a_1(t,\lambda)]r^{2-n}\leq\int_{t_0}^{t_1}(a_2(t,\lambda)-a_1(t,\lambda))\int_\Omega \nabla w_2\cdot\nabla w_1^*dx.$$
	In the same way, applying \eqref{fifi}, \eqref{l100c}-\eqref{l100d} and \eqref{l100r}, we obtain 
	$$\abs{\int_{t_0}^{t_1}\int_\Omega [\mathcal B(t,x,\lambda)\cdot\nabla w_2]w_1^*dxdt}\leq Cr^{3-n-\frac{1}{8}},$$
	with $C>0$ independent of $r$.
	From these two last estimates and the identity \eqref{l100e}, we deduce that
	$$c\inf_{t\in[t_0,t_1]}[a_2(t,\lambda)-a_1(t,\lambda)]r^{2-n}\leq Cr^{3-n-\frac{1}{8}},\quad r\in(0,1),$$
	with $c,C>0$ independent of $r$. This last identity clearly contradicts \eqref{t3f}. Therefore, \eqref{t3d} implies \eqref{t3e}. For $n=2$, we can prove the same result by choosing 
	$$\Phi_y(x)=\frac{1}{2\pi}\ln(|x-y|),\quad x\in\overline{\Omega}.$$
\end{proof}

\subsection{Recovery of the nonlinear convection term at $\xi=0$}
In all this section we fix $\lambda\in\R$. Our goal is to show that under the hypothesis of Theorem~\ref{t1} there holds,
\bel{t1f} \mathcal B_1(t,x,\lambda,0)=\mathcal B_2(t,x,\lambda,0),\quad (t,x)\in(0,T)\times\Omega.\ee
For this purpose, following the analysis of the preceding section and applying Proposition \ref{t3}, we can reduce this problem to an inverse problem for the IBVP
$$\left\{
\begin{array}{ll}
\partial_tw_j-a(t,\lambda) \Delta w_j+\mathcal B_j(t,x,\lambda,0)\cdot\nabla w_j=0  & \mbox{in}\ (0,T)\times\Omega ,
\\
w_j= h &\mbox{on}\ (0,T)\times\partial\Omega\\
w_j(0,x)=0 &\ x\in\Omega,
\end{array}
\right.$$
where the term $a(t,\lambda)$ is defined by
$$a(t,\lambda)=a_1(t,\lambda)=a_2(t,\lambda).$$
We associate with this problem the boundary map
$$\Lambda_{j,\lambda}:\mathcal H_0\ni h\mapsto \partial_\nu w_j|_{(0,T)\times\partial\Omega}$$
and, applying \eqref{t1c}, we obtain \eqref{t3d} and we want to prove that \eqref{t1f} holds true. This result can be deduced from  an extension of the analysis of \cite{CK1}. Namely, due to the presence of the time dependent second order coefficient $a(t,\lambda)$, we need to consider new class of geometric optics (GO in short) solutions. Following, the argumentation of \cite{CK1}, defining
$$B_j(t,x,\lambda)=\mathcal B_j(t,x,\lambda,0),\quad j=1,2,$$
we consider some class of solutions $w_j\in H^1(0,T;H^{-1}(\Omega))\cap L^2(0,T;H^{1}(\Omega))$ of the problems
\bel{eqGO1}\left\{
\begin{array}{ll}
	\partial_tw_1-a(t,\lambda) \Delta w_1 +B_1(t,x,\lambda)\cdot\nabla w_1=0  & \mbox{in}\ (0,T)\times\Omega ,
	\\
	w_1(0,x)=0 &\ x\in\Omega,
\end{array}
\right.\ee
\bel{eqGO2}\left\{
\begin{array}{ll}
	-\partial_tw_2-a(t,\lambda) \Delta w_2 -\textrm{div}(B_2(t,x,\lambda) w_2)=0  & \mbox{in}\ (0,T)\times\Omega ,
	\\
	w_2(T,x)=0 &\ x\in\Omega.
\end{array}
\right.\ee
Following \cite{CK1}, we fix $\rho>1$,  $\omega\in\mathbb S^{n-1}$, $\tau\in\R$, $\xi\in\omega^\bot$, $\delta\in(0,1)$ and we consider solutions of the form 
\[w_1=e^{\rho^2t+\frac{\rho}{\sqrt{a(t,\lambda)}}x\cdot\omega}\left[\left(1-e^{-\delta t}\right)b_{1,\rho}(t,x)\exp\left(-\frac{a'(t,\lambda)(x\cdot\omega)^2}{8a(t,\lambda)^2}\right)e^{-it\tau-ix\cdot\xi}+z_{1,\rho}(t,x)\right],\]
\[w_2=e^{-\rho^2t-\frac{\rho}{\sqrt{a(t,\lambda)}}x\cdot\omega}\left[\left(1-e^{-\delta (T-t)}\right)b_{2,\rho}(t,x)\exp\left(\frac{a'(t,\lambda)(x\cdot\omega)^2}{8a(t,\lambda)^2}\right)+z_{1,\rho}(t,x)\right]\]
of the problems \eqref{eqGO1}-\eqref{eqGO2}. 
Here, following \cite{CK1}, we define the functions $b_{j,\rho}$, $j=1,2$, in such a way that they satisfy
$$-2\sqrt{a}\,\omega\cdot \nabla b_{1,\rho}+(B_{1,\rho}\cdot \omega)b_{1,\rho}=0,\quad 2\sqrt{a}\,\omega\cdot \nabla b_{2,\rho}+(B_{2,\rho}\cdot \omega)b_{2,\rho}=0,$$
with $B_{j,\rho}$, some suitable smooth approximation of the function $B_j$. Finally, we choose the functions $$z_{j,\rho}\in H^1(0,T;H^{-1}(\Omega))\cap L^2(0,T;H^{1}(\Omega))$$ satisfying the following conditions
\bel{GO5}z_{1,\rho}(0,x)=0,\quad z_{2,\rho}(T,x)=0,\quad x\in\Omega\ee
as well as the decay estimate
$$\lim_{\rho\to+\infty}(\rho^{-1}\norm{z_{j,\rho}}_{L^2(0,T;H^{1}(\Omega))}+\norm{z_{j,\rho}}_{L^2((0,T)\times\Omega))})=0.$$
The construction of the GO solutions satisfying the above properties can be deduced by combining the arguments used in the proof of \cite{CK1} with a Carleman estimate similar to \cite[Proposition 3.1]{CK1} (see Proposition \ref{pp1}). Then, following \cite[Corollary 1.1]{CK1}, we deduce that \eqref{t1f} holds true.

\subsection{Recovery of the Taylor series of the nonlinear convection term at $\xi=0$}
Thus far we have shown that under the hypotheses of Theorem~\ref{t1} there holds:
\bel{diff_equal}
a_1=a_2=a \quad \text{on $(0,T)\times \R$}
\ee
and
\bel{conv_equal}
\mathcal B_1=\mathcal B_2=B\quad \text{on $(0,T)\times \Omega\times\R\times \{0\}.$}
\ee
Therefore, to conclude the proof of Theorem~\ref{t1} it suffices to show that 
\bel{b_taylor_eq}
\partial_\xi^\beta \mathcal B_1(t,x,\lambda,0)= \partial_\xi^\beta \mathcal B_2(t,x,\lambda,0),\quad (t,x,\lambda)\in (0,T)\times \Omega\times \R,\quad \beta \in (\N \cup\{0\})^n. 
\ee

Throughout the remainder of this section we will fix $\lambda \in \R$ and use an induction argument on the size of the multi-index
$$|\beta|=\beta_1+\ldots+\beta_n,$$ 
to show that under the hypotheses of Theorem~\ref{t1}, equation \eqref{b_taylor_eq} is satisfied. Observe that for $|\beta|=0$, there is nothing to prove as both sides of \eqref{b_taylor_eq} are equal, thanks to \eqref{conv_equal}. Next, let $m\in \N$ and let us assume for the hypothesis of our induction that \eqref{b_taylor_eq} is satisfied for all $|\beta|=0,\ldots,m-1$. We would like to prove that \eqref{b_taylor_eq} also holds for all multi-indices $\beta$ with $|\beta|=m$.

To this end, let us begin by noting that
\bel{B_j_form}
\mathcal B_j(t,x,\tau,\xi)=b_j(t,x,\tau,\xi)\,B(t,x,\tau),\quad \text{for $j=1,2$},
\ee
for some $B\in \mathcal C^{\infty}([0,T]\times \overline\Omega\times \R)^n$, and some $b_j\in \mathcal C^\infty([0,T]\times\overline{\Omega}\times\R\times\R^n) $ that satisfies \eqref{form1} with $b=b_j$. For $k=1,2,\ldots,m+1$, let $v_{k}\in H^1(0,T;\mathcal C^2(\overline\Omega))),$ be solutions to the equation
$$ \partial_t v_k -a(t,\lambda)\Delta v_k +B(t,x,\lambda)\cdot\nabla v_k=0,\quad \text{on $(0,T)\times \Omega$},$$
that additionally satisfy $v_k=0$ on $\{0\}\times\Omega$. Let $v_{m+2} \in H^1(0,T;\mathcal C^2(\overline\Omega)))$ be a solution to the adjoint equation
 $$ \partial_t v_{m+2} +a(t,\lambda)\Delta v_{m+2} +\textrm{div}\,(B(t,x,\lambda)\,v_{m+2})=0,\quad \text{on $(0,T)\times \Omega$},$$
 that additionally satisfies $v_{m+2}=0$ on $\{T\}\times\Omega$. Finally, for $k=1,\dots,m+2$, we define
 $$ g_k=v_k|_{(0,T)\times\partial\Omega}.$$
Next, let $s=(s_1,\dots,s_{m+1})$ be in a small neighborhood of the origin in $\R^{m+1}$ and for $j=1,2$, define $u_{j,s}$ to be the unique small solution to the equation

\[
\left\{
\begin{array}{ll}
	\partial_tu_{j,s}-\textrm{div}(a(t,u_{j,s}) \nabla u_{j,s})+\mathcal B_j(t,x,u_{j,s},\nabla u_{j,s})\cdot\nabla u_{j,s}=0  & \mbox{in}\ (0,T)\times\Omega:=M ,
	\\
	u_{j,s}=\lambda+s_1\,g_1+\ldots+s_{m+1}\,g_{m+1} &\mbox{on}\ (0,T)\times\partial\Omega:=\Sigma,\\
	u_{j,s}(0,x)=\lambda & \text{on $\Omega$}.
\end{array}
\right.
\]

Let 
$$ w^{(m+1)}_j= \frac{\partial^{(m+1)}\,u_{j,s}}{\partial s_1\ldots\partial s_{m+1}}|_{s=0}, \quad \text{for $j=1,2$}.$$
In view of Lemma~\ref{l5}, $w^{(m+1)}_j$ solves the boundary value problem \eqref{eq6} with $H^{(m+1)}$ replaced by $H_j^{(m+1)}$ where $H_j^{(m+1)}$ is given analogously to \eqref{Hm} with $b$ replaced by $b_j$ and $K^{(m+1)}$ replaced by $K_j^{(m+1)}$. Note also that in view of our induction assumption for $|\beta|\leq m-1$ and Lemma~\ref{l5} there holds:
\bel{K_iden} 
K^{(m+1)}_1=K^{(m+1)}_2.
\ee
Applying the condition \eqref{t1c} together with Lemma~\ref{l5} and our induction assumption for all $|\beta|\leq m-1$, it also follows that 
\bel{w_normal}
\partial_\nu w^{(m+1)}_1= \partial_\nu w^{(m+1)}_2 \quad \text{on $\Sigma$}.
\ee 
Next, recalling \eqref{K_iden} it follows that the function
$$w=w^{(m+1)}_1-w^{(m+2)}_2$$ 
satisfies
\begin{multline*}
	0=\partial_t w-a(t,\lambda)\Delta w+B(t,x,\lambda)\cdot\nabla w\\
	+\sum_{\ell\in\pi(m+1)}\sum_{j_1,\ldots,j_{m}=1}^n\left((\partial_{\xi_{j_1}}\ldots\partial_{\xi_{j_{m}}}(b_1-b_2)|_{\xi=0})\partial_{j_1}v_{\ell_1}\ldots\partial_{j_{m}} v_{\ell_{m}}\right)B\cdot\nabla v_{\ell_{m+1}}
\end{multline*}
Multiplying the latter equation with $v_{m+2}$ and integrating by parts on $(0,T)\times \Omega$ together with the fact that 
$$ w|_{t=0}=0,\quad \text{and}\quad w|_{\Sigma}=\partial_\nu w|_{\Sigma}=0,$$
it follows that 
\[\sum_{\ell_1,\ldots,\ell_{m+1}\in \pi(m+1)}\int_{(0,T)\times \Omega}\left(\sum_{j_1,\ldots,j_m=1}^nQ^{j_1,\ldots,j_m}\partial_{j_1}v_{\ell_1}\ldots \partial_{j_m}v_{\ell_m}\right) (B\cdot \nabla v_{\ell_{m+1}})\,v_{m+2}\,dt\,dx=0, 
\]
where the symmetric tensor $Q$ with elements $Q^{j_1,\ldots,j_m}$ is given by 
$$ Q^{j_1,\ldots,j_m}= \partial_{\xi_{j_1}}\ldots\partial_{\xi_{j_{m}}}(b_1-b_2)|_{\xi=0},\quad \text{on $(0,T)\times \Omega\times \R$,}$$
for all $j_1,\ldots,j_m=1,\ldots,n$. Finally, applying Proposition~\ref{prop1} with $B_0(\cdot)=B(\cdot,\lambda)$ and $a_0(\cdot)=a(\cdot,\lambda)$, we conclude that
\bel{}
(\partial_{\xi_{j_1}}\ldots\partial_{\xi_{j_{m}}}(b_1-b_2)|_{\xi=0})\, B =0 \quad \text{on $(0,T)\times\Omega\times \R$,}
\ee
for all $j_1,\ldots,j_m=1,\ldots,n$. Recalling \eqref{B_j_form}, this yields the desired claim \eqref{b_taylor_eq} for $|\beta|=m$. This concludes the induction and completes the proof of Theorem~\ref{t1}.

\section{Geometric Optic solutions with higher regularity}
\label{sec_smoothgo}
\subsection{Principal part of smoother GO}

Our proof of Proposition~\ref{prop1} will partly rely on the construction of Geometric Optics solutions with higher regularity. More precisely, we will consider GO solutions to the equation
\bel{eqGO1sm}\left\{
\begin{array}{ll}
	\partial_tw_1-a_0(t) \Delta w_1 +B_0(t,x)\cdot\nabla w_1=0  & \mbox{in}\ (0,T)\times\Omega ,
	\\
	w_1(0,x)=0 &\ x\in\Omega,
\end{array}
\right.\ee
as well as GO solutions for 
\bel{eqGO2sm}\left\{
\begin{array}{ll}
	-\partial_tw_2-a_0(t) \Delta w_2 -\textrm{div}(B_0(t,x) w_2)=0  & \mbox{in}\ (0,T)\times\Omega ,
	\\
	w_2(T,x)=0 &\ x\in\Omega,
\end{array}
\right.\ee
that lie in the energy space  $H^1(0,T;\mathcal C^2(\overline{\Omega}))$.
We present in this section a canonical construction of these GO solutions that will depend on a large asymptotic parameter $\rho$ with $|\rho|\gg1$ and formally concentrates on a ray in $\Omega$ that passes through a point $x_0\in\Omega$ in a fixed direction $\omega\in \mathbb S^{n-1}$. 

We fix $N_1=[\frac{n}{2}]+5$ and consider solutions of the form
\bel{GO7}w_1(t,x)=U_{+,\rho}(t,x)=e^{\rho^2t+\frac{\rho}{\sqrt{a_0(t)}}x\cdot\omega}\left[\underbrace{\sum_{\ell=0}^{N_1}c_{+,\ell}(t,x)\,\rho^{-\ell}}_{V_{+,\rho}}+R_{+,\rho}(t,x)\right],\ee
and
\bel{GO8}w_2(t,x)=U_{-,\rho}(t,x)=e^{-\rho^2t-\frac{\rho}{\sqrt{a_0(t)}}x\cdot\omega}\left[\underbrace{\sum_{\ell=0}^{N_1}c_{-,\ell}(t,x)\,\rho^{-\ell}}_{V_{-,\rho}}+R_{-,\rho}(t,x)\right],\ee
to equations \eqref{eqGO1sm} and \eqref{eqGO2sm} respectively. The principal terms $V_{\pm,\rho}$ will be constructed canonically and will be localized near the ray passing through $x_0$ in the direction $\omega$ while the correction terms $R_{\pm,\rho}$ will converge to zero as $|\rho|$ approaches infinity.
  
Let  $\tilde{B}_0$ be a smooth  compactly supported extension of the function $B_0$ to $\R\times\R^n$. We define $L_\pm$, $P_{\rho,\pm}$ the differential operators given by
\bel{l}L_{+}v=\partial_tv-a_0(t)\Delta v +\tilde{B}_0\cdot\nabla v,\quad L_{-}v=-\partial_tv-a_0(t)\Delta v -\textrm{div}(\tilde{B}_0v),\ee
and
\bel{P_pm}
P_{\rho,\pm}v= e^{\mp(\rho^2t+\frac{\rho}{\sqrt{a_0(t)}}x\cdot\omega)}L_{\pm}(e^{\pm(\rho^2t+\frac{\rho}{\sqrt{a_0(t)}}x\cdot\omega)}v) 
\ee
The GO solutions will be constructed by applying the WKB method to the conjugated operator $P_{\rho,\pm}$. It is straightforward to see that
$$P_{\rho,\pm}v=\rho\, J_{\pm}v + L_{\pm}v,$$
where  
\bel{j1}J_+v:=-2\sqrt{a_0(t)}\omega\cdot \nabla v+\left[\frac{\tilde{B}_0(t,x)}{\sqrt{a_0(t)}}\cdot \omega-\frac{a'_0(t)}{2a_0(t)^{\frac{3}{2}}}x\cdot\omega\right]v,\ee
\bel{j2}J_-v:=2\sqrt{a_0(t)}\omega\cdot \nabla v+\left[\frac{\tilde{B}_0(t,x)}{\sqrt{a_0(t)}}\cdot \omega-\frac{a'_0(t)}{2a_0(t)^{\frac{3}{2}}}x\cdot\omega\right]v.\ee
We choose $c_{\pm,\ell}$, $\ell=1,\ldots N_1$, in such a way that 
\bel{GO9}J_+c_{+,0}=0,\quad J_-c_{-,0}=0\ee
and, for $\ell=1,\ldots N_1$,
\bel{GO10}J_+c_{+,\ell}=-L_+c_{+,\ell-1},\quad J_-c_{-,\ell}=-L_-c_{-,\ell-1}.\ee
Let $\zeta\in\mathcal C^\infty_c((0,T))$. Define the functions 
\bel{GO13} e_{\pm}(t,x)=\exp\left(\frac{\mp a_0'(t)(x\cdot\omega)^2}{8a_0^2(t)}\right)\exp\left(\frac{\mp\int_0^{+\infty} \tilde{B}_0(x+s\omega,t)\cdot\omega \,ds}{2\sqrt{a_0(t)}}\right),\ee
Then, for any smooth function $d$ solving the transport equation
\bel{GO14}\omega\cdot \nabla d(t,x)=0,\quad (t,x)\in \R^{1+n}\ee
we can define
\bel{GO15}
c_{\pm,0}(t,x)=\zeta(t)\,e_\pm(t,x)\,d(t,x),\quad \forall \, (t,x)\in (0,T)\times \Omega. \ee
In fact as we would like our GO solutions to concentrate near a fixed ray passing through a point $x_0\in \Omega$ in the direction of $\omega$, we will make a canonical choice for the function $d$ as follows. Let $\delta \in (0,1)$ and let $\alpha_1,\ldots,\alpha_{n-1}$ be unit vectors such that the set 
$$\{\omega,\alpha_1,\ldots,\alpha_{n-1}\}$$
forms an orthonormal basis in $\R^n$. We set
\bel{d_canon} 
d(t,x)=\prod_{j=1}^{n-1} \chi_0\left(\frac{(x-x_0)\cdot\alpha_j}{\delta}\right),
\ee
where $\chi_0:\R\to [0,1]$ is a smooth function with $\chi_0(t)=1$ for $|t|\leq \frac{1}{2}$ and $\chi(t)=0$ for $|t|\geq 1$. It is clear that $d$ solves \eqref{GO14}. 

Using the fact that $c_{\pm,0}\in \mathcal C^\infty(\R\times\R^n)$, $a_0\in\mathcal C^\infty([0,T])$ and $\tilde{B}_0\in\mathcal C^\infty(\R\times\R^n)^n$, we can choose
the solution $c_{\pm,\ell}$, $\ell=1,\ldots N_1$ to the equations \eqref{GO9}-\eqref{GO10} to be lying in    $\mathcal C^\infty(\R\times\R^n)$. Also, using the fact that the functions $c_{\pm,0}(t,x)$ are supported away from $t=0$ and $t=T$ and the fact that the transport equations in \eqref{GO10} are independent of the time variable $t$ (in terms of the derivatives), we can prove, by using cut-off functions in time, that the solutions of these equations can be chosen in such a way that
\bel{GO16}c_{+,\ell}(0,x)=c_{-,\ell}(T,x)=0,\quad x\in\Omega.\ee
In order to complete our construction of Geometric Optic solutions, we need to show that it is possible to construct the remainder terms 
$$R_{\pm,\rho}\in H^1(0,T;\mathcal C^2(\overline{\Omega}))$$ 
satisfying the decay property
\bel{GO17}\norm{R_{\pm,\rho}}_{L^\infty(0,T;\mathcal C^2(\overline{\Omega}))}\leq C\,|\rho|^{-1}\ee
with $C>0$ independent of $\rho$ as well as the final and initial condition
\bel{GO18}R_{+,\rho}(0,x)=R_{-,\rho}(T,x)=0,\quad x\in\Omega.\ee
\subsection{Remainder terms}
In this subsection, we will complete the construction of GO solving \eqref{eqGO1sm}-\eqref{eqGO2sm} of the form \eqref{GO7}-\eqref{GO8}  lying in $H^1(0,T;\mathcal C^2(\overline{\Omega}))$ with remainder terms $R_{\pm,\rho}$ satisfying the decay properties \eqref{GO17}-\eqref{GO18}. For this purpose, following \cite{CK1,CK} we will use Carleman estimates in negative order Sobolev space. Let us consider two parameters $s$,$\rho$ with $|\rho|>s>1$, and define the perturbed weight
\bel{phi}\phi_{\pm,s}(t,x):=\pm \left(\rho^2t+\frac{\rho}{\sqrt{a_0(t)}}\omega\cdot x\right)-\frac{s}{\sqrt{a_0(t)}}{((x+x_1)\cdot\omega)^2\over 2}.\ee
We set
$$P_{\rho,\pm,s}:=e^{-\phi_{\pm,s}}L_{\pm}e^{\phi_{\pm,s}}.$$
Here $x_1\in\R^n$ is chosen in such a way that 
\bel{x0}x_1\cdot\omega=2+\sup_{x\in\Omega}|x|.\ee
Following \cite[Proposition 3.1]{CK1}, and recalling the notation $M=(0,T)\times\Omega$ and $\Sigma=(0,T)\times\partial\Omega$, we can prove the following Carleman estimate.
\begin{prop}\label{pp1} There exist $s_1>1$ and, for $s>s_1$,  $\rho_1(s)$ such that for any $v\in\mathcal C^2(\overline{M})$ satisfying the condition  
	\begin{equation}\label{t2a}v_{\vert \Sigma}=0,\quad v_{\vert t=0}=0,\end{equation}
	the estimate
	\bel{lll1a} \begin{aligned}&|\rho|\int_{\Sigma_{+,\omega}} |\partial_\nu v|^2|\omega\cdot\nu| d\sigma(x)dt+s|\rho| \int_\Omega|v|^2(T,x)dx+s^{-1}\int_M|\Delta v|^2dxdt+s\rho^2\int_M|v|^2dxdt\\
		&\leq C\left[\norm{P_{\rho,+,s}v}^2_{L^2(M)}+|\rho|\int_{\Sigma_{-,\omega}} |\partial_\nu v|^2|\omega\cdot\nu| d\sigma(x)dt\right]\end{aligned}\ee
	holds true for $s>s_1$, $|\rho|\geq \rho_1(s)$  with $C$  depending only on  $\Omega$, $T$, $k$, $a_0$ and $B_0$.
	Moreover, there exist  $s_2>1$ and, for $s>s_2$,  $\rho_2(s)$ such that for all $v\in\mathcal C^2(\overline{M})$ satisfying the condition \begin{equation}\label{t2c}v_{\vert \Sigma}=0,\quad v_{\vert t=T}=0,\end{equation}
	the estimate
	\bel{lll1b} \begin{aligned}&|\rho|\int_{\Sigma_{-,\omega}} |\partial_\nu v|^2|\omega\cdot\nu| d\sigma(x)dt+s|\rho| \int_\Omega|v|^2(x,0)dx+s^{-1}\int_M|\Delta v|^2dxdt+s\rho^2\int_M|v|^2dxdt\\
		&\leq C\left[\norm{P_{\rho,-,s}v}^2_{L^2(M)}+|\rho|\int_{\Sigma_{+,\omega}} |\partial_\nu v|^2|\omega\cdot\nu| d\sigma(x)dt\right]\end{aligned}\ee
	holds true for $s>s_2$, $|\rho|\geq \rho_2(s)$. Here $s_1$, $\rho_1$, $s_2$ and $\rho_2$ depend only on  $\Omega$, $T$, $k$, $a_0$ and $B_0$. 
\end{prop}

We will now apply Proposition \ref{pp1} for deriving two Carleman estimates in Sobolev space of negative order.  In a similar way to \cite{CK1,Ki3}, for all $m\in\R$, we introduce the space $H^m_\rho(\R^{n})$ defined by
\[H^m_\rho(\R^{n})=\{u\in\mathcal S'(\R^{n}):\ (|\xi|^2+\rho^2)^{m\over 2}\hat{u}\in L^2(\R^{n})\},\]
with the norm
\[\norm{u}_{H^m_\rho(\R^{n})}^2=\int_{\R^n}(|\xi|^2+\rho^2)^{m}|\hat{u}(\xi)|^2 d\xi .\]
For all tempered distributions $u\in \mathcal S'(\R^{n})$, we denote here  by $\hat{u}$ the Fourier transform of $u$ which, for $u\in L^1(\R^{n})$, is defined by
$$\hat{u}(\xi):=\mathcal Fu(\xi):= (2\pi)^{-{n\over2}}\int_{\R^{n}}e^{-ix\cdot \xi}u(x)dx.$$
From now on, for $m\in\R$ and $\xi\in \R^n$,  we set $$\left\langle \xi,\rho\right\rangle=(|\xi|^2+\rho^2)^{1\over2}$$
and $\left\langle D_x,\rho\right\rangle^m u$ defined by
\[\left\langle D_x,\rho\right\rangle^m u=\mathcal F^{-1}(\left\langle \xi,\rho\right\rangle^m \mathcal Fu).\]
For $m\in\R$ we define also the class of symbols
\[S^m_\rho=\{c_\rho\in\mathcal C^\infty(\R\times\R^{n}\times\R^{n}):\ |\pd_t^k\pd_x^\alpha\pd_\xi^\beta c_\rho(t,x,\xi)|\leq C_{k,\alpha,\beta}\left\langle \xi,\rho\right\rangle^{m-|\beta|},\  \alpha,\beta\in\mathbb N^n,\ k\in\mathbb N\}.\]
Following \cite[Theorem 18.1.6]{Ho3}, for any $m\in\R$ and $c_\rho\in S^m_\rho$, we define $c_\rho(t,x,D_x)$, with  $D_x=-i\nabla_x$, by
\[c_\rho(t,x,D_x)z(x)=(2\pi)^{-{n\over 2}}\int_{\R^{n}}c_\rho(t,x,\xi)\hat{z}(\xi)e^{ix\cdot \xi} d\xi,\quad z\in\mathcal S(\R^n).\]
For all $m\in\R$, we set also $OpS^m_\rho:=\{c_\rho(t,x,D_x):\ c_\rho\in S^m_\rho\}$.
We fix
$$P_{\rho,\pm}\cdot:=e^{\mp (\rho^2t+\frac{\rho}{\sqrt{a_0(t)}} x\cdot\omega)}L_{\pm}(e^{\pm (\rho^2t+\frac{\rho}{\sqrt{a_0(t)}} x\cdot\omega)}\cdot)$$
and we consider the following Carleman estimate. 

\begin{prop}\label{p1} There exists $\rho_2'>\rho_2$, depending only on $\Omega$, $T$ $k$, $a_0$ and $B_0$,  such that  for all $v\in \mathcal C^1([0,T];\mathcal C^\infty_0(\Omega))$ satisfying $v_{|t=T}=0$ we have 
	\bel{p1a}\norm{v}_{L^2(0,T; H^{1-N_1}_\rho(\R^{n}))}\leq C\norm{P_{\rho,-}v}_{L^2(0,T;H^{-N_1}_\rho(\R^{n}))},\quad |\rho|>\rho_2',\ee
	with $C>0$ depending on $\Omega$, $T$, $k$, $a_0$ and $B_0$. In the same way, for all $y\in \mathcal C^1([0,T];\mathcal C^\infty_0(\Omega))$ satisfying $y_{|t=0}=0$ we have 
	\bel{p1b}\norm{y}_{L^2(0,T; H^{1-N_1}_\rho(\R^{n}))}\leq C\norm{P_{\rho,+}y}_{L^2(0,T;H^{-N_1}_\rho(\R^{n}))},\quad |\rho|>\rho_2',\ee
	with $C>0$ depending on $\Omega$, $T$, $a_0$ and $B_0$.
\end{prop}
\begin{proof} 
	We will only give the proof of \eqref{p1a}, the proof of \eqref{p1b} being similar.
	We fix $\phi_{\rho,s}$, given by \eqref{phi}, and  we  consider
	$$P_{\rho,-,s}:=e^{-\phi_{-,s}}L_{-}e^{\phi_{-,s}}$$
	and we decompose $P_{\rho,-,s}$ into three terms
	\[P_{\rho,-,s}=P_{1,-}+P_{2,-}+P_{3,-},\]
	with
	$$\begin{aligned}P_{1,-}=&-a_0(t)\Delta+2\rho s((x+x_1)\cdot\omega+s^2((x+x_1)\cdot\omega)^2-\sqrt{a_0(t)}s\\
	&-\frac{a'_0(t)}{2a_0(t)^{\frac{3}{2}}}[\rho\omega\cdot x)-s{((x+ x_1)\cdot\omega)^2\over 2}],\end{aligned}$$
	$$ P_{2,-}=-\pd_t-2\sqrt{a_0(t)}[\rho-s((x+x_1)\cdot\omega)]\omega\cdot\nabla+2s,$$
	\[ P_{3,-}=\tilde{B}_0\cdot\nabla-\frac{(\rho+s((x+x_1)\cdot\omega))}{\sqrt{a_0(t)}}\tilde{B}_0\cdot\omega.\]
	We pick $ \tilde{\Omega}$  a bounded open and smooth set of $\R^n$ such that $\overline{\Omega}\subset\tilde{\Omega}$. In order to prove \eqref{p1a},  we fix $w\in \mathcal C^1([0,T];\mathcal C^\infty_0(\tilde{\Omega}))$ satisfying $w_{|t=T}=0$ and we  consider the quantity
	\[\left\langle D_x,\rho\right\rangle^{-N_1}(P_{1,-}+P_{2,-})\left\langle D_x,\rho\right\rangle^{N_1} w.\]
	In this formula, for any $z\in \mathcal C^1([0,T];\mathcal C^\infty_0(\tilde{\Omega}))$ we define 
	\[\left\langle D_x,\rho\right\rangle^m z(t,x)=\mathcal F^{-1}_x(\left\langle \xi,\rho\right\rangle^m \mathcal F_xz(t,\cdot))(x).\]
	with the partial Fourier transform $\mathcal F_x$  defined by
	$$\mathcal F_xz(t,\xi):= (2\pi)^{-{n\over2}}\int_{\R^{n}}e^{-ix\cdot \xi}z(t,x)dx.$$
	From now on,  $C>0$ denotes a generic constant depending on $\Omega$, $T$, $k$, $a_0$ and $B_0$.
	The properties of composition of pseudo-differential operators (e.g. \cite[Theorem 18.1.8]{Ho3}) implies that
$$\left\langle D_x,\rho\right\rangle^{-N_1}(P_{1,-}+P_{2,-})\left\langle D_x,\rho\right\rangle^{N_1}=P_{1,-}+P_{2,-}+K_\rho(t,x,D_x),$$
	where $K_\rho$ is given by
	\[K_\rho(t,x,\xi)=\nabla_\xi\left\langle \xi,\rho\right\rangle^{-N_1}\cdot D_x(p_{1,-}(t,x,\xi)+p_{2,-}(t,x,\xi))\left\langle \xi,\rho\right\rangle^{N_1}+\underset{\left\langle \xi,\rho\right\rangle\to+\infty}{ o}(1),\]
	with
	$$\begin{aligned}p_{1,-}(t,x,\xi)=&-a_0(t)|\xi|^2+2\rho s((x+x_1)\cdot\omega+s^2((x+x_1)\cdot\omega)^2-\sqrt{a_0(t)}s\\
	&-\frac{a_0'(t)}{2a_0(t)^{\frac{3}{2}}}[\rho\omega\cdot x-s{((x+ x_1)\cdot\omega)^2\over 2}],\end{aligned}$$
	$$ p_{2,-}(t,x,\xi)=-i2\sqrt{a_0(t)}[\rho-s(((x+x_1)\cdot\omega)]\omega\cdot\xi+2s.$$
	Thus, one can check that
	\bel{l2c} \norm{K_\rho(t,x,D_x)w}_{L^2((0,T)\times \R^n)}\leq Cs^2\norm{w}_{L^2((0,T)\times \R^n)}.\ee
	On the other hand,  applying \eqref{lll1b} to $w$ with $M$ replaced by  $\tilde{M}=(0,T)\times\tilde{\Omega}$,  we get
	$$\norm{P_{1,-}w+P_{2,-}w}_{L^2((0,T)\times\R^{n})}\geq C\left(s^{-1/2}\norm{\Delta w}_{L^2((0,T)\times\R^{n})}+s^{1/2}|\rho|\norm{ w}_{L^2((0,T)\times\R^{n})}\right)$$
	and, choosing ${|\rho|\over s^2}$ sufficiently large, it follows
	\bel{l2c'}
	\norm{P_{1,-}w+P_{2,-}w}_{L^2((0,T)\times\R^{n})}\geq Cs^{1/2}\norm{ w}_{L^2(0,T;H^1_\rho(\R^{n}))}.
	\ee
	Combining this estimate with \eqref{l2c}, for ${|\rho|\over s^2}$ sufficiently large, we obtain
	\bel{l2d}\begin{array}{l} \norm{(P_{1,-}+P_{2,-})\left\langle D_x,\rho\right\rangle^{N_1} w}_{L^2(0,T;H^{-N_1}_\rho(\R^{n}))}\\
		=\norm{\left\langle D_x,\rho\right\rangle^{-N_1}(P_{1,-}+P_{2,-})\left\langle D_x,\rho\right\rangle^{N_1} w}_{L^2((0,T)\times \R^n)}\\ \geq  Cs^{1/2}\norm{ w}_{L^2(0,T;H^1_\rho(\R^{n}))}.\end{array}\ee
	Moreover, we have
	\begin{equation}\label{l2f}\begin{aligned}&\norm{P_{3,-}\left\langle D_x,\rho\right\rangle^{N_1} w}_{L^2(0,T;H^{-N_1}_\rho(\R^{n}))}\leq\\
		&\norm{\tilde{B}_0\cdot\nabla \left\langle D_x,\rho\right\rangle^{N_1} w}_{L^2(0,T;H^{-N_1}_\rho(\R^{n}))}+\norm{\frac{(\rho+s((x+x_0)\cdot\omega))\tilde{B}_0\cdot\omega}{\sqrt{a_0(t)}}\left\langle D_x,\rho\right\rangle^{N_1} w}_{L^2(0,T;H^{-N_1}_\rho(\R^{n}))}\\
		&\leq C\norm{\tilde{B}_0}_{W^{1+N_1,\infty}((0,T)\times\R^n)^n}\left(\norm{\nabla \left\langle D_x,\rho\right\rangle^{N_1} w}_{L^2(0,T;H^{-N_1}_\rho(\R^{n}))}+\norm{\left\langle D_x,\rho\right\rangle^{N_1} w}_{L^2(0,T;H^{-N_1}_\rho(\R^{n}))}\right)\\
		&\leq C\norm{\left\langle D_x,\rho\right\rangle^{N_1+1} w}_{L^2(0,T;H^{-N_1}_\rho(\R^{n}))}\\
		&\leq C\norm{ w}_{L^2(0,T;H^{1}_\rho(\R^{n}))}.\end{aligned}
		\end{equation}
	In view of \eqref{l2d}-\eqref{l2f}, we deduce that, fixing $s>1$ sufficiently large, we can find $C>0$ independent of $\rho$ such that
	\bel{p1l}\norm{P_{\rho,-,s}\left\langle D_x,\rho\right\rangle^{N_1} w}_{L^2(0,T;H^{-N_1}_\rho(\R^{n}))}\geq C\norm{w}_{L^2(0,T;H^1_\rho(\R^{n}))}.\ee
	Let $\psi_0\in\mathcal C^\infty_0(\tilde{\Omega})$ be such that $\psi_0=1$ on $\overline{\Omega_1}$, with  $\Omega_1$ an open neighborhood of $\overline{\Omega}$ such that $\overline{\Omega_1}\subset\tilde{\Omega}$. We fix $w(t,x)=\psi_0(x) \left\langle D_x,\rho\right\rangle^{-N_1} v(t,x)$ and  repeating the arguments used at the end of the proof of \cite[Proposition 4.1.]{CK1}, we deduce that \eqref{p1l} implies \eqref{p1a}.\end{proof}

Applying the estimate \eqref{p1a}-\eqref{p1b}, we are in position to complete the construction of the remainder terms $R_{\pm,\rho}$, satisfying the decay property \eqref{GO17}. For this purpose, we recall that $P_{\rho,\pm}=L_{\pm}+\rho J_\pm$ with $L_\pm$ and $J_\pm$ defined by \eqref{l}-\eqref{j2}. Then,  according to \eqref{GO9}-\eqref{GO10}, we have 
$$J_+c_{+,0}=0,\quad J_+c_{+,\ell+1}=-L_+c_{1,\ell},\quad \ell=1,\ldots N_1-1.$$
It follows
$$\begin{aligned}&L_+\left[e^{\rho^2t+\frac{\rho}{\sqrt{a_0(t)}}x\cdot\omega}\left(\sum_{\ell=0}^{N_1}c_{+,\ell}\rho^{-\ell}\right)\right]\\
&=e^{\rho^2t+\frac{\rho}{\sqrt{a_0(t)}}x\cdot\omega}P_{\rho,+}\left(\sum_{\ell=0}^{N_1}c_{+,\ell}\rho^{-\ell}\right)=e^{\rho^2t+\frac{\rho}{\sqrt{a_0(t)}}x\cdot\omega}\rho^{-N_1}L_{+}c_{+,N_1}.\end{aligned}$$
Therefore, the condition $L_{+}w_1=0$ is  fulfilled if and only if $R_{+,\rho}$
solves
$$P_{\rho,+}R_{+,\rho}(t,x)=-\rho^{-N_1}L_+c_{+,N_1}(t,x),\quad (t,x)\in(0,T)\times\Omega.$$
Thus, fixing $\phi_1\in\mathcal C^\infty_0(\R^n)$, such that $\phi_1=1$ on $\overline{\Omega}$, and
$$F(t,x)=-\phi_1(x)L_+c_{+,N_1}(t,x)$$
we can consider $R_{+,\rho}$ as a solution of 
\bel{CGO12}P_{\rho,+}R_{+,\rho}(t,x)=\rho^{-N_1}F(t,x),\quad (t,x)\in(0,T)\times\Omega.\ee
We fix $\tilde{\Omega}$ a smooth bounded open set  of $\R^n$ such that $\overline{\Omega}\subset \tilde{\Omega}$.
Applying the Carleman estimate \eqref{p1a},  we define the linear form $\mathcal G_\rho$ on $\{P_{\rho,-}z:\ z\in\mathcal C^1([0,T];\mathcal C^\infty_0(\tilde{\Omega})),\ z_{|t=T}=0\}$, considered as a subspace of $L^2(0,T;H^{-N_1}_\rho(\R^{n}))$ by
\[\mathcal K_\rho(P_{\rho,-}z)=\rho^{-N_1}\left\langle F,z\right\rangle_{L^2((0,T)\times\R^n))},\quad z\in\mathcal C^1([0,T];\mathcal C^\infty_0(\tilde{\Omega})),\ z_{|t=T}=0.\]
Then, \eqref{p1a} implies that, for all $z\in\mathcal C^1([0,T];\mathcal C^\infty_0(\tilde{\Omega}))$ satisfying $z_{|t=T}=0$, we have
\[\begin{aligned}|\mathcal K_\rho(P_{\rho,-}z)|&\leq \rho^{-N_1}\norm{F}_{L^2(0,T;H^{N_1-1}_\rho(\R^n))}\norm{z}_{L^2(0,T;H^{1-N_1}_\rho(\R^n))}\\
\ &\leq C\rho^{-1}\norm{c_{+,N_1}}_{H^1(0,T;H^{N_1+1}(\tilde{\Omega}))}\norm{P_{\rho,-}z}_{L^2(0,T;H^{-N_1}_\rho(\R^{n}))}.\end{aligned}\]
Thus, by the Hahn--Banach theorem we can extend $\mathcal K_\rho$ to a continuous linear form on $L^2(0,T;H^{-N_1}_\rho(\R^{n}))$ still denoted by $\mathcal K_\rho$ and satisfying 
$$\norm{\mathcal K_\rho}\leq C\norm{c_{+,N_1}}_{H^1(0,T;H^{N_1+1}(\tilde{\Omega}))}|\rho|^{-1}.$$ 
Therefore, there exists $R_{+,\rho}\in L^2(0,T;H^{N_1}_\rho(\R^{n}))$ such that 
\[\left\langle h,R_{+,\rho}\right\rangle_{L^2(0,T;H^{-N_1}_\rho(\R^{n})),L^2(0,T;H^{N_1}_\rho(\R^{n}))}=\mathcal K_\rho(h),\quad h\in L^2(0,T;H^{-N_1}_\rho(\R^{n})).\]
Fixing $h=P_{\rho,-}z$ with $z\in \mathcal C^\infty_0(M)$, we deduce that $R_{+,\rho}$ satisfying $P_{\rho,+}R_{+,\rho}=\rho^{-N_1}F$ in $M$. In addition, using the fact that
$$\partial_tR_{+,\rho}=-(P_{\rho,+}R_{+,\rho}-\partial_tR_{+,\rho})+\rho^{-N_1}F\in L^2(0,T;H^{N_1-2}(\Omega)),$$
we obtain $R_{+,\rho}\in H^1(0,T;H^{N_1-2}(\Omega))$. Moreover, fixing $h=P_{\rho,-}z$ with $z\in\mathcal C^\infty([0,T];\mathcal C^\infty_0(\tilde{\Omega}))$, $z_{|t=T}=0$ and allowing $z_{|t=0}$ to be arbitrary proves that $R_{+,\rho}(0,x)=0$ for  $x\in\Omega$. Therefore, $R_{+,\rho}$ fulfills condition \eqref{GO17} and, combining this with \eqref{GO16}, we deduce that $w_1$ given by \eqref{GO7} is lying in $H^1(0,T;H^{N_1-2}(\Omega))$ and it satisfies the condition \eqref{eqGO1sm}. In order to complete the construction of $w_1$, we only need to prove that $R_{+,\rho}$ satisfies the decay property \eqref{GO17}. For this purpose, applying the Sobolev embedding theorem we get
$$\begin{aligned}\norm{R_{+,\rho}}_{L^\infty(0,T;C^2(\overline{\Omega}))}&\leq C\norm{R_{+,\rho}}_{H^1(0,T;H^{N_1-2}(\Omega))}\\
&\leq C(\norm{R_{+,\rho}}_{L^2(0,T;H^{N_1}(\Omega))}+\norm{\rho^{-N_1}F}_{L^2(0,T;H^{N_1-2}(\Omega))})\\
\ &\leq C(\norm{R_{+,\rho}}_{L^2(0,T;H^{N_1}_\rho(\R^n))}+\rho^{-1})\leq C(\norm{\mathcal K_\rho}+|\rho|^{-1})=C|\rho|^{-1}.\end{aligned}$$
This proves that $R_{+,\rho}$ fulfills the decay \eqref{GO17}. Using similar arguments we can build $w_2$ given by \eqref{GO8} with $R_{-,\rho}$ satisfies the decay property \eqref{GO17}.

\section{Proof of Proposition~\ref{prop1}}

\subsection{Proof of Proposition~\ref{prop1} in the case $m=1$}
In the case $m=1$, the tensor $Q=(Q^1,\dots,Q^n)$ is a vector. In this case the integral identity in the statement of the proposition reduces to 
\bel{m1_eq}
\int_{M} \left[(Q\cdot \nabla v_1)(B_0\cdot \nabla v_2)+(Q\cdot \nabla v_2)(B_0\cdot \nabla v_1) \right]\,v_3\,dt\,dx=0,
\ee
where $M=(0,T)\times \Omega$ and $v_{1},v_2\in H^1(0,T;\mathcal C^2(\overline\Omega))$ are any solutions to \eqref{eqGO1sm} and $v_{3}\in H^1(0,T;\mathcal C^2(\overline\Omega))$ is any solution to \eqref{eqGO2sm}. 

Let us fix $t_0\in (0,T)$ and $x_0\in \Omega$. We will prove the proposition by showing that $Q\otimes B_0$ vanishes at $(t_0,x_0)$. Let $\omega_1,\omega_2 \in \mathbb S^{n-1}$ be unit vectors that satisfy 
\bel{omega_ortho}
\omega_1\cdot\omega_2=0.
\ee
We start by defining, for any $\rho>1$, the functions 
$$v_1(t,x)=U^{(1)}_{+,\rho}(t,x),\quad \text{and}\quad v_2(t,x)=U^{(2)}_{+,\rho}(t,x),\quad \text{for all $(t,x)\in M$},$$
where, following the construction in Section~\ref{sec_smoothgo}, $U^{(j)}_{+,\rho}$, $j=1,2$ is the canonical GO solution to \eqref{eqGO1sm}, given by \eqref{GO7} with $\omega=\omega_j$, that concentrates on the ray passing through the point $x_0$ in the direction $\omega_j$.
Next, we define the unit vector
\bel{omega_3}
\omega_3 = \frac{1}{\sqrt{2}}(\omega_1+\omega_2)\in \mathbb S^{n-1},
\ee
and set
$$ v_3(t,x)= U^{(3)}_{-,\sqrt{2}\rho}(t,x),\quad \text{for all $(t,x)\in M$},$$
to be the canonical GO solution to \eqref{eqGO2sm}, given by \eqref{GO8} with $\omega=\omega_3$, that concentrates on the ray passing through the point $x_0$ in the direction $\omega_3$. Recall that for $j=1,2$,
$$ v_j=U^{(j)}_{+,\rho}= e^{\rho^2t+\frac{\rho}{\sqrt{a_0(t)}}x\cdot\omega_j}(V^{(j)}_{+,\rho}+R_{+,\rho}^{(j)}),$$
and that
$$v_3=U^{(3)}_{-,\sqrt{2}\rho}= e^{-2\rho^2t-\frac{\sqrt{2}\rho}{\sqrt{a_0(t)}}x\cdot\omega_3}(V^{(3)}_{-,\sqrt{2}\rho}+R^{(3)}_{-,{\sqrt{2}\rho}}).
$$
Next, let us write
$$S_{\rho}= \int_{M} \left[(Q\cdot \nabla v_1)(B_0\cdot \nabla v_2)+(Q\cdot \nabla v_2)(B_0\cdot \nabla v_1) \right]\,v_3\,dt\,dx,$$
with $v_1$, $v_2$ and $v_3$ as chosen above and note that $S_\rho=0$ by \eqref{m1_eq}. On the other hand, observe by applying \eqref{omega_3} that the exponential terms in the expression for the product $v_1v_2v_3$ cancel out. The same principle also holds for products
$$(Q\cdot \nabla v_1)(B_0\cdot \nabla v_2)v_3 \quad\text{and}\quad (Q\cdot \nabla v_2)(B_0\cdot \nabla v_1)v_3.$$ 
Using this observation together with the expressions \eqref{GO7}--\eqref{GO8} and the error estimate \eqref{GO17} it follows that
\bel{S_eq_lim}
0=\lim_{\rho\to\infty} \rho^{-2}S_\rho=\int_M a_0^{-1}\left[(\omega_1\cdot Q)(\omega_2\cdot B_0)+(\omega_2\cdot Q)(\omega_1\cdot B_0)\right]\,c_{+,0}^{(1)}\,c_{+,0}^{(2)}\,c_{-,0}^{(3)}\,dt\,dx.
\ee
We recall that $c^{(j)}_{\pm,0}$, $j=1,2,3,$ are as given by
$$c^{(j)}_{\pm,0}(t,x)=\zeta(t)\,e_\pm^{(j)}(t,x)\,d^{(j)}(t,x),\quad \forall \, (t,x)\in M,$$
where $e_{\pm}^{(j)}$ are strictly positive functions defined by \eqref{GO13} with $\omega=\omega_{j}$ and $d^{(j)}$ are as given by
$$ 
d^{(j)}(t,x)=\prod_{k=1}^{n-1} \chi_0\left(\frac{(x-x_0)\cdot\alpha_k^{(j)}}{\delta}\right),
$$
 where $\alpha_1^{(j)},\ldots,\alpha_{n-1}^{(j)}$ are unit vectors such that $\{\omega_{j},\alpha_1^{(j)},\ldots,\alpha_{n-1}^{(j)}\}$ forms an orthonormal basis of $\R^n$. Next, we set
 \bel{zeta_def_chi}
 \zeta(t) =\chi_0\left(\frac{t-t_0}{\delta}\right),
 \ee
 and assume that $\delta\in (0,1)$ is sufficiently small so that $\delta<\max\{t_0,T-t_0\}$.
 As the unit vectors $\omega_1$ and $\omega_2$ are orthogonal, it is straightforward to see that  the product 
 $$c_{+,0}^{(1)}(t,x)\,c_{+,0}^{(2)}(t,x)\,c_{-,0}^{(3)}(t,x)$$
 is supported in a $\sqrt{3}\,\delta$ neighborhood of the point $(t_0,x_0)$. As the functions $e_{\pm}^{(j)}$, $j=1,2,3$ are positive, it follows that given any continuous function $f$ on $M$ there holds
 $$ \lim_{\delta\to 0} \delta^{-(n+1)}\int_{M}f(t,x)\,c_{+,0}^{(1)}(t,x)\,c_{+,0}^{(2)}(t,x)\,c_{-,0}^{(3)}(t,x)\,dt\,dx=C_0f(t_0,x_0),$$
 for some non-zero constant $C_0$ that depends only on $M$, $a_0$ and $B_0$. Thus, by multiplying the right hand side of equation \eqref{S_eq_lim} with $\delta^{-(n+1)}$ and taking the limit as $\delta$ approaches zero we deduce that
\bel{m_1_den}
(\omega_1\cdot Q(t_0,x_0))\,(\omega_2\cdot B_0(t_0,x_0))+(\omega_2\cdot Q(t_0,x_0))\,(\omega_1\cdot B_0(t_0,x_0))=0,
\ee
for any pair of orthogonal unit vectors $\omega_1$ and $\omega_2$ in $\R^n$.

Note that if $B_0(t_0,x_0)=0$, there is nothing to prove. So we assume that $B_0(t_0,x_0)$ is a non-zero vector. Let
\bel{xi_def}
\xi = |B_0(t_0,x_0)|^{-1}\,B_0(t_0,x_0)\in \mathbb S^{n-1},
\ee
and let $\omega \in \mathbb S^{n-1}$ be orthogonal to $\xi$. Setting $\omega_1=\xi$ and $\omega_2=\omega$, it follows from \eqref{m_1_den} that
$$ \omega\cdot Q(t_0,x_0)=0\quad \text{for any $\omega \in \mathbb S^{n-1}$ with $\omega\cdot\xi=0$}.$$
Therefore the vectors $Q(t_0,x_0)$ and $B(t_0,x_0)$ are co-linear. Thus, the two terms in equation \eqref{m_1_den} are identical implying that 
$$ (\omega_1\cdot Q(t_0,x_0))\,(\omega_2\cdot B_0(t_0,x_0))=0,$$
for any pair of orthogonal unit vectors $\omega_1$ and $\omega_2$ in $\R^n$. As $B_0(t_0,x_0)\neq 0$, it is straightforward to conclude that $Q(t_0,x_0)=0.$ This concludes the proof in the case $m=1$.

\subsection{Proof of Proposition~\ref{prop1} in the case $m\geq 2$}
 
Let us fix $(t_0,x_0)\in M$ and let $\omega_1,\omega_2 \in \mathbb S^{n-1}$ satisfy
\bel{omega_cond}
\omega_1\cdot \omega_2 \notin \{-1,0,1\}.
\ee
Next, let 
$$\kappa= -\frac{m-1}{2\,\omega_1\cdot\omega_2},\quad \text{and}\quad \tilde\kappa= \sqrt{m+\kappa^2},$$
and define the unit vector $\omega_3 \in \mathbb S^{n-1}$ via
$$\omega_3=\frac{m\,\omega_1+\kappa\,\omega_2}{\tilde\kappa}.$$
Let us define for each $\rho>1$, the functions
$$ v_1=\ldots=v_m=U_{+,\rho}^{(1)}=e^{\rho^2t+\frac{\rho}{\sqrt{a_0(t)}}x\cdot\omega_1}(V^{(1)}_{+,\rho}+R_{+,\rho}^{(1)}),$$
to be the canonical GO solutions to \eqref{eqGO1sm} constructed in Section~\ref{sec_smoothgo} that concentrate along the ray in $\Omega$ passing through the point $x_0$ in the direction $\omega_1$. We set
$$v_{m+1}=U_{+,\kappa\rho}^{(2)}=e^{\kappa^2\rho^2t+\frac{\kappa\rho}{\sqrt{a_0(t)}}x\cdot\omega_2}(V^{(2)}_{+,\kappa\rho}+R_{+,\kappa\rho}^{(2)}),$$
to be the canonical GO solution to \eqref{eqGO1sm} constructed in Section~\ref{sec_smoothgo} that concentrates along the ray in $\Omega$ passing through the point $x_0$ in the direction $\omega_2$ and finally
$$v_{m+2}=U^{(3)}_{-,\tilde\kappa\rho}=e^{-\tilde\kappa^2\rho^2t-\frac{\tilde\kappa\rho}{\sqrt{a_0(t)}}x\cdot\omega_3}(V^{(3)}_{-,\tilde\kappa\rho}+R_{-,\tilde\kappa\rho}^{(3)}),$$
to be the canonical GO solution to \eqref{eqGO2sm} constructed in Section~\ref{sec_smoothgo} that concentrates along the ray in $\Omega$ passing through the point $x_0$ in the direction $\omega_3$.  Let
\bel{S_rho}
S_{\rho}=\sum_{\ell \in \pi(m+1)}\int_{(0,T)\times \Omega}\left(\sum_{j_1,\ldots,j_m=1}^nQ^{j_1,\ldots,j_m}\partial_{j_1}v_{\ell_1}\ldots \partial_{j_m}v_{\ell_m}\right) (B_0\cdot \nabla v_{\ell_{m+1}})\,v_{m+2}\,dt\,dx,
\ee
with $v_1,\ldots, v_{m+2}$ as chosen above and note that $S_\rho=0$ by the hypothesis of the proposition. On the other hand, in view of the definitions of $\kappa$, $\tilde \kappa$ and $\omega_3$, it follows that the exponential terms in the expression for the product $v_1\,v_2\ldots v_{m+2}$ cancel out. The same principle also holds for products
$$ Q^{j_1\ldots j_m} \partial_{j_1} v_{\ell_1}\, \partial_{j_2}v_{\ell_2}\ldots \partial_{j_{m}}v_{\ell_m} (B_0\cdot \nabla v_{\ell_{m+1}})v_{m+2},$$
for any $\ell \in \pi(m+1)$ and any $j_1,\ldots,j_m=1,\ldots,n$. Thus, by recalling the expressions \eqref{GO7}--\eqref{GO8} and the error estimate \eqref{GO17} it also follows that
\bel{S_eq_lim_final}
0=\lim_{\rho\to\infty} \rho^{-(m+1)}S_\rho=m!\,\kappa\int_M a_0(t)^{-\frac{m+1}{2}} K_{\omega_1,\omega_2}(t,x)\,F(t,x)\,dt\,dx,
\ee
where the scalar function $K_{\omega_1,\omega_2}\in \mathcal C(\overline M)$ is given by the expression
\[
K_{\omega_1,\omega_2} = m\, Q(\underbrace{\omega_1,\ldots,\omega_1}_{\text{$m-1$ times}},\omega_2)(B_0\cdot \omega_1)+Q(\omega_1,\ldots,\omega_1)(B_0\cdot \omega_2).
\]
and the smooth function $F\in \mathcal C^{\infty}(\overline M)$ is defined by
$$ F(t,x) = (c^{(1)}_{+,0}(t,x))^{m}\, c^{(2)}_{+,0}(t,x)\,c_{-,0}^{(3)}(t,x).$$
Recall that $c^{(j)}_{\pm,0}$, $j=1,2,3$ are given by
$$c^{(j)}_{\pm,0}(t,x)=\zeta(t)\,e_\pm^{(j)}(t,x)\,d^{(j)}(t,x),\quad \forall \, (t,x)\in M,$$
where $e_{\pm}^{(j)}$ are strictly positive functions defined by \eqref{GO13} with $\omega=\omega_{j}$ and the functions $d^{(j)}$, $j=1,2,3$ are given by
$$ 
d^{(j)}(t,x)=\prod_{k=1}^{n-1} \chi_0\left(\frac{(x-x_0)\cdot\alpha_k^{(j)}}{\delta}\right),
$$
 where $\alpha_1^{(j)},\ldots,\alpha_{n-1}^{(j)}$ are unit vectors such that $\{\omega_{j},\alpha_1^{(j)},\ldots,\alpha_{n-1}^{(j)}\}$ forms an orthonormal basis of $\R^n$. Next and analogously to the previous section, we define $\zeta(t)$ via \eqref{zeta_def_chi}. Recall from \eqref{omega_cond} that $\omega_1\neq \pm\omega_2$. Therefore, we have $$\textrm{Span}\,\{\alpha_1^{(1)},\ldots,\alpha_{n-1}^{(1)},\alpha_1^{(2)},\ldots,\alpha_{n-1}^{(2)}\}=\R^n$$ 
 and thus the function $F$ is supported in a $2\sqrt{n}\,\delta$ neighborhood of the point $(t_0,x_0)$. As the functions $e_{\pm}^{(j)}$, $j=1,2,3$ are positive, it follows analogously to the previous section that given any $f\in \mathcal C(\overline M)$, there holds
 $$\lim_{\delta\to 0}\delta^{-(n+1)}\int_M f(t,x) \,F(t,x)\,dt\,dx=C_0\,f(t_0,x_0),$$
 for some non-zero $C_0$ only depending on $M$, $a_0$ and $B_0$. Hence, by multiplying the right hand side of equation \eqref{S_eq_lim_final} with $\delta^{-(n+1)}$ and taking the limit as $\delta$ approaches zero we deduce that $K_{\omega_1,\omega_2}(t_0,x_0)=0$ for any pair of $\omega_1,\omega_2$ that satisfy \eqref{omega_cond}. In fact since $K_{\omega_1,\omega_2}$ depends continuously on $\omega_1$ and $\omega_2$ we can deduce, by continuity, that $K_{\omega_1,\omega_2}(t_0,x_0)=0$ for all unit vectors $\omega_1,\omega_2$ (that is to say, including the cases $\omega_1\cdot\omega_2\in\{\pm 1,0\}$).  In other words, given any $(t_0,x_0)\in M$, and any pair of unit vectors $\omega_1$, $\omega_2$, there holds
\bel{m_2_den}
 m\, Q(\underbrace{\omega_1,\ldots,\omega_1}_{\text{$m-1$ times}},\omega_2)(B_0\cdot \omega_1)+Q(\omega_1,\ldots,\omega_1)(B_0\cdot \omega_2)=0,
\ee
where the left hand side expression is evaluated at the point $(t_0,x_0)\in M$. 

\subsubsection{The case $m=2$}
When $m=2$, equation \eqref{m_2_den} is sufficient to deduce that $Q\otimes B_0$ must vanish at the point $(t_0,x_0)$. To show this, we may assume without loss of generality that $B_0(t_0,x_0)\neq 0$. Let the unit vector $\xi$ be defined by \eqref{xi_def}. Applying \eqref{m_2_den} with $\omega_1=\omega_2=\xi$, it follows that 
$$Q(\xi,\xi)=0 \quad \text{at $(t_0,x_0)$.}$$
Applying \eqref{m_2_den} with $\omega_1=\xi$ and $\omega_2=\omega$ any unit vector orthogonal to $\xi$ it follows that
$$Q(\xi,\omega)=0\quad \text{at $(t_0,x_0)$}.$$
Finally, applying \eqref{m_2_den} with $\omega_1=\omega$ any unit vector orthogonal to $\xi$ and $\omega_2=\xi$, it follows that
$$Q(\omega,\omega)=0\quad \text{at $(t_0,x_0)$}.$$
Together with the symmetry of $Q$, it follows immediately from the last three identities that $Q(t_0,x_0)=0$. This completes the proof of the proposition in the case $m=2$ since $(t_0,x_0)\in M$ is arbitrary.

\subsubsection{The case $m\geq 3$}
We keep the identity \eqref{m_2_den} for now and return to the statement of the proposition and define an alternative choice for the test functions $v_1,\ldots,v_{m+2}$. To this end, let us fix $s \in \{2,\ldots,m-1\}$ and define the positive number $\hat{\kappa}$ (depending on the value of $s$) by
\bel{kappa_def}
\hat{\kappa}=\sqrt{\frac{(s-1)^2+(s-1)}{(m+1-s)^2-(m+1-s)}}.
\ee
We let $\rho>1$ and set
$$ v_1=v_2=\ldots=v_{s-1} = U_{+,\rho}^{(1)}=e^{\rho^2t+\frac{\rho}{\sqrt{a_0(t)}}x\cdot\omega_1}(V^{(1)}_{+,\rho}+R_{+,\rho}^{(1)}),$$
and 
$$ v_s=U_{+,-(s-1)\rho}^{(1)}=e^{(s-1)^2\rho^2t-\frac{(s-1)\rho}{\sqrt{a_0(t)}}x\cdot\omega_1}(V^{(1)}_{+,-(s-1)\rho}+R_{+,-(s-1)\rho}^{(1)})$$
to be the canonical GO solutions to \eqref{eqGO1sm} constructed in Section~\ref{sec_smoothgo} that concentrate along the ray in $\Omega$ passing through the point $x_0$ in the direction $\omega_1$. Next, we define
$$ v_{s+1}=\ldots=v_{m+1} = U_{+,\hat\kappa\rho}^{(2)}=e^{\hat\kappa^2\rho^2t+\frac{\hat\kappa\rho}{\sqrt{a_0(t)}}x\cdot\omega_2}(V^{(2)}_{+,\hat\kappa\rho}+R_{+,\hat\kappa\rho}^{(2)}),$$
to be the canonical GO solutions to \eqref{eqGO1sm} constructed in Section~\ref{sec_smoothgo} that concentrate along the ray in $\Omega$ passing through the point $x_0$ in the direction $\omega_2$. Finally, we define
$$v_{m+2}=U_{-,\hat\kappa(m-s+1)\rho}^{(2)}=e^{-\hat\kappa^2(m-s+1)^2\rho^2t-\frac{\hat\kappa(m-s+1)\rho}{\sqrt{a_0(t)}}x\cdot\omega_2}(V^{(2)}_{-,\hat\kappa(m-s+1)\rho}+R_{-,\hat\kappa(m-s+1)\rho}^{(2)}).$$
to be the canonical GO solution to \eqref{eqGO2sm} constructed in Section~\ref{sec_smoothgo} that concentrates along the ray in $\Omega$ passing through the point $x_0$ in the direction $\omega_2$. 

In view of \eqref{kappa_def} it follows that the exponential terms in the expression for the product $v_1\,v_2\ldots v_{m+2}$ cancel out. The same principle also holds for products
$$ Q^{j_1\ldots j_m} \partial_{j_1} v_{\ell_1}\, \partial_{j_2}v_{\ell_2}\ldots \partial_{j_{m}}v_{\ell_m} (B_0\cdot \nabla v_{\ell_{m+1}})v_{m+2},$$
for any $\ell \in \pi(m+1)$ and any $j_1,\ldots,j_m=1,\ldots,n$. Thus, defining $S_\rho$ analogously to \eqref{S_rho} corresponding to the current choice of the test functions $v_1,\ldots,v_{m+2}$, it follows from \eqref{GO7}--\eqref{GO8} and remainder estimates \eqref{GO17} that 
\bel{S_limit_1}
0=\lim_{\rho\to\infty}\rho^{-(m+1)}S_\rho=
-m!\,(s-1)\hat{\kappa}^{m-s+1}\,\int_M a_0(t)^{-\frac{m+1}{2}}\,K_{s,\omega_1,\omega_2}(t,x)F(t,x)\,dt\,dx,
\ee
where 
$$K_{s,\omega_1,\omega_2}= s\, Q(\underbrace{\omega_1,\ldots,\omega_1}_{\text{$s-1$ times}},\underbrace{\omega_2,\ldots,\omega_2}_{\text{$m+1-s$ times}})(B_0\cdot \omega_1)+(m-s+1)Q(\underbrace{\omega_1,\ldots,\omega_1}_{\text{$s$ times}},\underbrace{\omega_2,\ldots,\omega_2}_{\text{$m-s$ times}})(B_0\cdot \omega_2),$$
and 
$$F(t,x)=(c^{(1)}_{+,0}(t,x))^{s}\,(c^{(2)}_{+,0}(t,x))^{m-s+1}\,c^{(2)}_{-,0}(t,x).$$
Analogously to the previous section we set $\zeta$ as in \eqref{zeta_def_chi} and multiply the right hand side of \eqref{S_limit_1} with $\delta^{-(n+1)}$ and take the limit $\delta\to 0$ to deduce that given any $s=2,\ldots,m-1$, any $(t_0,x_0)\in M$, and any pair of unit vectors $\omega_1$, $\omega_2$, there holds
\bel{m_3_den}
s\, Q(\underbrace{\omega_1,\ldots,\omega_1}_{\text{$s-1$ times}},\underbrace{\omega_2,\ldots,\omega_2}_{\text{$m+1-s$ times}})(B_0\cdot \omega_1)+(m-s+1)Q(\underbrace{\omega_1,\ldots,\omega_1}_{\text{$s$ times}},\underbrace{\omega_2,\ldots,\omega_2}_{\text{$m-s$ times}})(B_0\cdot \omega_2)=0,
\ee
where the left hand side expression is evaluated at the point $(t_0,x_0)$. Combining \eqref{m_3_den} with \eqref{m_2_den} we deduce that \eqref{m_3_den} actually holds for all $s=2,3,\ldots,m$. 

In order to conclude the proof of the proposition when $m\geq 3$, we begin by fixing $(t_0,x_0)\in M$ and proceed to prove that $Q\otimes B_0$ vanishes at $(t_0,x_0)$. Observe that if $B_0(t_0,x_0)$ is zero then the claim is trivial, so we will make the standing assumption that $B_0(t_0,x_0)$ is a non-zero vector and aim to prove that $Q(t_0,x_0)$ is the zero tensor. Let us define the unit vector $\xi$ by \eqref{xi_def} and return to the identity \eqref{m_3_den} evaluated at the point $(t_0,x_0)$. Setting $s=2$, $\omega_1=\omega_2=\xi$ in \eqref{m_3_den} it follows that
 \bel{final_0}
 Q(\xi,\xi,\ldots,\xi)=0,\quad \text{at $(t_0,x_0)$}.
 \ee
Next, setting $s=2,\ldots,m$ (Recall that we can also set $s=m$, thanks to \eqref{m_2_den}), $\omega_1=\xi$ and $\omega_2=\omega$ any unit vector orthogonal to $\xi$, it follows from \eqref{m_3_den} that
\bel{final_1}
Q(\underbrace{\xi,\ldots,\xi}_{\text{$s-1$ times}},\underbrace{\omega,\ldots,\omega}_{\text{$m+1-s$ times}})=0,\quad \text{at $(t_0,x_0)$},
\ee
for all $\omega\in \mathbb S^{n-1}$ that satisfies $\omega\cdot \xi=0$ and any $s=2,\ldots,m$. Finally, returning to \eqref{m_3_den} again and plugging $s=m$, $\omega_1=\omega$ any unit vector orthogonal to $\xi$ and $\omega_2=\xi$ it follows that
\bel{final_2}
Q(\omega,\omega,\ldots,\omega)=0,\quad \text{at $(t_0,x_0)$},
\ee
for all $\omega\in \mathbb S^{n-1}$ that satisfies $\omega\cdot \xi=0$. It is clear from \eqref{final_0} and \eqref{final_1}--\eqref{final_2} together with symmetry of $Q$ that the tensor $Q$ must vanish at the point $(t_0,x_0)$, thus concluding the proof.

\section*{Acknowledgments}
A.F acknowledges support from the Fields institute for research in mathematical sciences. The work of Y.K. is partially supported by  the French National Research Agency ANR (project MultiOnde) grant ANR-17-CE40-0029.  The research of G.U. is partially supported by NSF, a Walker Professorship at UW and a Si-Yuan Professorship at IAS, HKUST.

\bigskip
\vskip 1cm

\end{document}